\documentclass[10pt]{article}
\usepackage{geometry} % see geometry.pdf on how to lay out the page. There's lots.
\pdfoutput=1
\usepackage{graphicx,url}
\usepackage{amssymb,amsmath, amsthm}
\usepackage{epstopdf}
\usepackage{subfig}
\newtheorem{theorem}{Theorem}
\newtheorem{remark}{Remark}[section]

\title{Flying randomly in $\mathbb{R}^d$ with Dirichlet displacements}
\numberwithin{equation}{section}
\author{ Alessandro De Gregorio\footnote{alessandro.degregorio@uniroma1.it}, 
Enzo Orsingher \footnote{enzo.orsingher@uniroma1.it}	\\
Dipartimento di Scienze Statistiche\\
``Sapienza", University of Rome\\
P.le Aldo Moro, 5 - 00185, Rome, Italy}
\allowdisplaybreaks
\begin{document}

\maketitle
\begin{abstract}

Random flights in $\mathbb{R}^d,d\geq 2,$ with Dirichlet-distributed displacements and uniformly distributed orientation are analyzed. The explicit characteristic functions of the position $\underline{\bf X}_d(t),\,t>0,$ when the number of changes of direction is fixed are obtained. The probability distributions are derived by inverting the characteristic functions for all dimensions $d$ of $\mathbb{R}^d$ and many properties of the probabilistic structure of $\underline{\bf X}_d(t),t>0,$ are examined.

If the number of changes of direction is randomized by means of a fractional Poisson process, we are able to obtain explicit distributions for $P\{\underline{\bf X}_d(t)\in d\underline{\bf x}_d\}$ for all $d\geq 2$. A Section is devoted to random flights in $\mathbb{R}^3$ where the general results are discussed.

The existing literature is compared with the results of this paper where in our view the classical Pearson's problem of random flights is resolved by suitably randomizing the step lengths. The random flights where changes of direction are governed by a homogeneous Poisson process are analyzed and compared with the model of Dirichlet-distributed displacements of this work.
\\

{\it Key words: Bessel functions, Dirichlet distributions, fractional Poisson process, Mittag-Leffler functions, hyperspherical coordinates, Random flights, Struve functions, telegraph and wave equations, Wigner law.}
\end{abstract}
\section{Introduction}
The problem of random flights has been appealing for many researchers in different scientific fields. The original formulation is due to the statistician Karl Pearson, who, in a brief letter quoted in {\it Nature}, 1905, wrote: ``A man starts from a point $O$ and walks $a$ yards in a straight line; he then turns
through any angle whatever and walks another $a$ yards in a second straight line. He
repeats this process $n$ times. I require the probability that after $n$ of these stretches
he is at distance between $r$ and $r+\delta r$ from his starting point $O$.''

Pearson's aim was that of modelling the random migration of mosquitos invading cleared jungle regions, while Rayleigh in the same issue of {\it Nature}, observed that Pearson's problem is equivalent to the problem of the superposition of $n$ sound vibrations with unit amplitude and arbitrary phase. The Pearson walk was generalized by Kluyver (1905) who considered steps with arbitrary but deterministic length. Successively, Rayleigh (1919) extended isotropic planar random flights to the space $\mathbb{R}^3$, useful as a possible model of statistical mechanics of a diluted solution of polymeric chains. These random models also emerge in astronomy to describe the stellar dynamics as noted by Chandrasekhar (1943). The author also pointed out the link between random flights and diffusion processes. More recently, Stadje (1987) and Masoliver {\it et al.} (1993) dealt with a two-dimensional random walk moving with constant velocity and with directions uniformly distributed in $[0,2\pi]$.

Over the years many papers, particularly in the physical literature, analyzed the properties of these random models, see Section 2 in Hughes (1995) and references therein.

For the position $\underline{\bf X}_d(t),$ reached at time $t>0$ by the random flights in $\mathbb{R}^d,d\geq 2,$ the conditional distribution
\begin{equation}\label{eq:introd}
P\{\underline{\bf X}_d(t)\in d\underline{{\bf x}}_d|\mathcal{N}_d(t)=n\}
\end{equation}
has been the main object of investigation, where $\mathcal{N}_d(t),t>0,$ is the number of changes of direction recorded up to time $t$. In the first part of the paper the number $\mathcal{N}_d(t),t>0,$ of changes of direction up to time $t$ is assumed to be a fixed number $n$. Franceschetti (2007) has obtained a condition for \eqref{eq:introd} to be uniformly distributed and this relates the dimension $d$ of the space in which  $\underline{\bf X}_d(t),\,t>0,$ develops and the number $n$ of changes of direction. A similar work has been carried out by Garcia-Pelayo (2008) and discussed by Le Caer (2010) who also considered random flights with Dirichlet distributed displacements. Orsingher and De Gregorio (2007) tackled the problem of  random flights in higher spaces by dealing with uniformly hyperspherical distributions of the orientation  of motion. Beghin and Orsingher (2010) considered a planar random motion where the deviations occur at odd-order Poisson events (or at even-Poisson events thus implicitly assuming that the displacements take a Dirichlet distribution). 

In the original Pearson's formulation of the problem of random flights the length of steps was deterministic and the probability distribution of $\underline{\bf X}_d(t),\,t>0,$ lead to integrals which could not be explicitly worked out (see Watson, 1922, pag.421). For random flights in $\mathbb{R}^d$ with uniformly distributed steps it was possible to obtain the distribution \eqref{eq:introd} for $d=2,4,$ and for arbitrary values of $n$ (and also the non-conditional distributions).

We remark that in all papers mentioned above the deviations are separated by exponentially distributed time lapses. This corresponds to assuming instants of changes of direction uniformly distributed under the condition that the number of Poisson events $N(t)$ is fixed. This basic assumption permits us to obtain the explicit distribution of the position of the moving particle only in the spaces $\mathbb{R}^2$ and $\mathbb{R}^4$. In $\mathbb{R}^2$ the distribution was obtained by recursive arguments by Stadje (1987), by Masoliver {\it et al.} (1993). In $\mathbb{R}^4$ the explicit distribution was obtained in Orsingher and De Gregorio (2007). This unlucky circumstance (for concrete purposes the space $\mathbb{R}^3$ is clearly the most important one) is here overcome by the assumption that the intervals between successive changes of direction exhibit a Dirichlet distribution.

By choosing a suitable basic parameter of the Dirichlet random displacements for each $d$, we are able here to obtain the explicit distribution \eqref{eq:introd} of $\underline{\bf X}_d(t),\,t>0,$ for all $n$. Furthermore, the distribution has the universal isotropic form
\begin{equation}\label{eq:introd2}
h(\underline{\bf x}_d)=A(c^2t^2-||\underline{\bf x}_d||^2)^b
\end{equation}
where $b$ depends on $n$ and $d$, while $A$ is the necessary normalizing factor (depending on $t$). In our view this solves the classical Pearson's problem of random flights for all values of $d$ in Euclidean spaces. We show below that functions of the form \eqref{eq:introd2} are solution to some $d$-dimensional telegraph equation. 

By resorting to fractional Poisson processes we can randomize the distribution with respect to $n$, thus arriving at the non-conditional distributions of $\underline{\bf X}_d(t),\,t>0,$ for all $d\geq 2$. Furthermore, we can extract the distribution in $\mathbb{R}^2$ and $\mathbb{R}^4$ known so far as particular cases of those derived here with Dirichlet distributed steps.

We now describe the random flights analyzed in this paper. A random walker, starting from the origin of a frame of reference, moves in $d$-dimensional real space, with $d\geq 2$, at finite speed (denoted by $c$) according to the following rules. We assume that in the time interval $[0,t]$, $n$ changes of direction of motion are recorded. We suppose that the instants at which the random walker changes direction are $0<t_1<t_2<\cdots<t_n<t$, $n\geq 1$, and denote the length of time separating these instants by $\tau_j=t_j-t_{j-1}$, $1\leq j\leq n+1$ with $t_0=0,\,t_{n+1}=t$. Each displacement has an orientation defined in $\mathbb{R}^d$ by the angles $(\theta_1,\theta_2,...,\theta_{d-2},\phi)$ and we suppose that $0\leq \theta_j\leq \pi$, $0\leq \phi\leq 2\pi$ with joint law equal to
\begin{equation}\label{eq:jointdis1}
g(\theta_1,....,\theta_{d-2},\phi)=\frac{\Gamma(\frac d2)}{2\pi^{\frac d2}}\sin^{d-2}\theta_1\sin^{d-3}\theta_2\cdots \sin\theta_{d-2}. 
\end{equation}
The law \eqref{eq:jointdis1} tells us that the direction is uniformly chosen on the hypersphere of $\mathbb{R}^d$ with radius one. An important assumption is that the random vector $(\tau_1,...,\tau_n)$ (representing the length of the $n$ displacements) possesses joint density equal to
\begin{equation}\label{eq:jointdis2}
f_1(\tau_1,...,\tau_n)=\frac{\Gamma((n+1)(d-1))}{(\Gamma(d-1))^{n+1}}\frac{1}{t^{(n+1)(d-1)-1}}\prod_{j=1}^{n+1}\tau_j^{d-2},
\end{equation}
where $0<\tau_j<t-\sum_{k=0}^{j-1}\tau_k$, $1\leq j\leq n$, and $\tau_{n+1}=t-\sum_{j=1}^n\tau_j$. The distribution $\eqref{eq:jointdis2}$ is a rescaled Dirichlet distribution, with parameters $(d-1,...,d-1),\, d\geq 2$. The probability distribution \eqref{eq:jointdis2} can be obtained as a marginal integral of the multiple uniform law as shown in Lachal {\it et al.} (2006).

We will treat also the random flights with intermediate step length having joint distribution
\begin{equation}\label{eq:jointdis3}
f_2(\tau_1,...,\tau_n)=\frac{\Gamma((n+1)(\frac d2-1))}{(\Gamma(\frac d2-1))^{n+1}}\frac{1}{t^{(n+1)(\frac d2-1)-1}}\prod_{j=1}^{n+1}\tau_j^{\frac d2-2}.
\end{equation}
where $0<\tau_j<t-\sum_{k=0}^{j-1}\tau_k$, $1\leq j\leq n$, and $\tau_{n+1}=t-\sum_{j=1}^n\tau_j$, which is a Dirichlet distribution (suitably rescaled) with parameters $(\frac d2-1,....,\frac d2-1)$, where $d\geq 3$.

The model treated here consists of the triple $(\underline{\theta},\underline{\tau},\mathcal{N}_d(t))$ of independent vectors where $\underline{\theta}=(\theta_1,....,\theta_{d-2},\phi)$ is the orientation of displacements (with uniform law \eqref{eq:jointdis1}), $\underline{\tau}=(\tau_1,...,\tau_n)$ represents the displacements and $\mathcal{N}_d(t)$ is the number of changes of orientation. In the models analyzed in Stadje (1987) and Orsingher and De Gregorio (2007), $\underline{\theta}$ has law coinciding with \eqref{eq:jointdis1},  $\underline{\tau}$ is uniformly distributed and $\mathcal{N}_d(t)$ is a homogeneous Poisson process.

Random processes with intertimes with non-uniform distribution have occasionally been considered in the literature (see Di Crescenzo, 2002 and  Pogorui and Rodriguez-Dagnino, 2005). Random motions in $\mathbb{R}^c$ with $n+1$ direction have been analyzed by Samoilenko (2001), Lachal (2006) and Lachal {\it et al}. (2006). 

Recalling that the motion develops at constant velocity $c$, the process described by the $d$-dimensional random flight $\underline{\bf X}_d(t)=(X_1(t),...,X_d(t)),t>0,$ has components equal to

\begin{align}\label{polar}
&X_d(t)=c\sum_{j=1}^{n+1}\tau_j
\sin\theta_{1,j}\sin\theta_{2,j}\cdot\cdot\cdot\sin\theta_{d-2,j}\sin\phi_{j}\notag\\
&X_{d-1}(t)=c\sum_{j=1}^{n+1}\tau_j
\sin\theta_{1,j}\sin\theta_{2,j}\cdot\cdot\cdot\sin\theta_{d-2,j}\cos\phi_{j}\notag\\
&\cdot\cdot\cdot\\
&X_2(t)=c\sum_{j=1}^{n+1}\tau_j
\sin\theta_{1,j}\cos\theta_{2,j}\notag\\
&X_1(t)=c\sum_{j=1}^{n+1}\tau_j \cos\theta_{1,j}.\notag
\end{align}

  \begin{figure}[t]
  \begin{center}
\includegraphics[angle=0,width=0.67\textwidth]{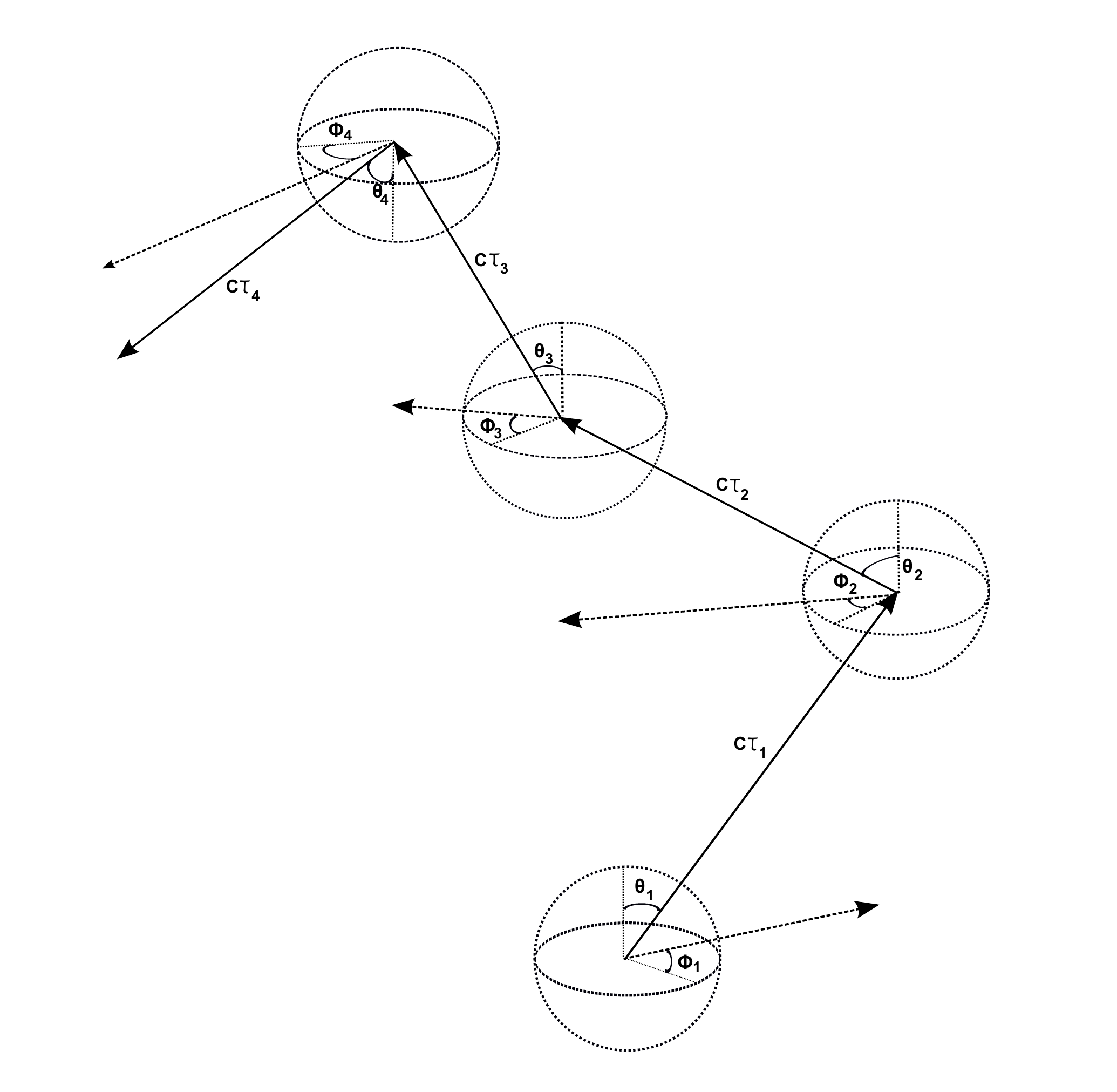}
\caption{A sample path of a random flight in $\mathbb{R}^3$ consisting of four displacements with angles $(\theta_j,\phi_j)$ defining each direction.}
\end{center}
\end{figure}

  \begin{figure}[t]
  \begin{center}
\includegraphics[angle=0,width=0.48\textwidth]{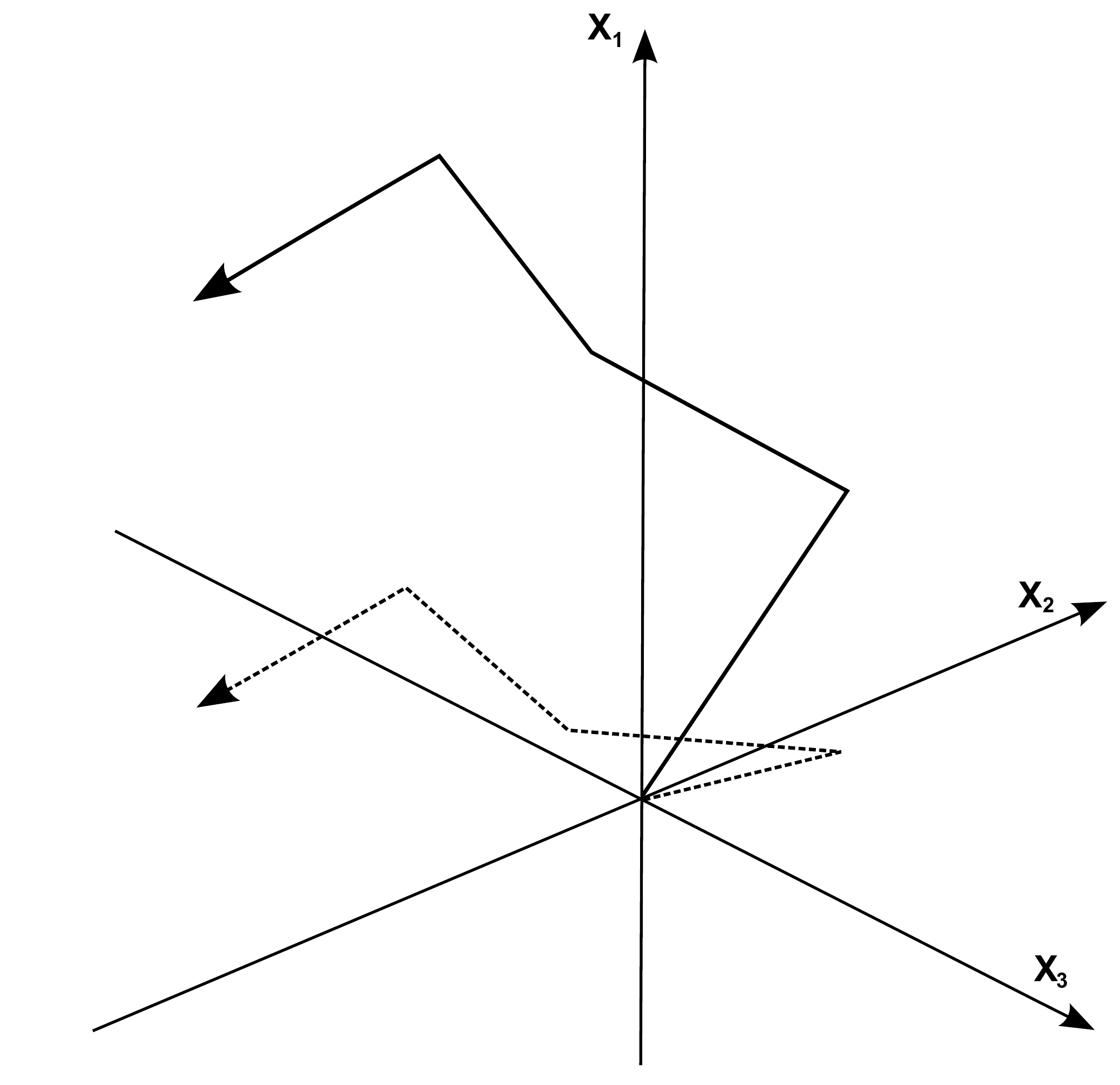}
\caption{The sample path of Figure 1 with its projection on $\mathbb{R}^2$ (dotted line).}
\end{center}
\end{figure}

If we consider a random flight in $\mathbb{R}^d$ with time intervals between successive changes of orientation distributed as \eqref{eq:jointdis2}, we will show that the the position of the moving point $\underline{\bf X}_d(t)=(X_1(t),...,X_d(t))$ at time $t$, has distribution 
\begin{equation}\label{eq:condlawint}
p_{\underline{\bf X}_d }(\underline{\bf x}_d,t;n) =\frac{\Gamma(\frac{n+1}{2}(d-1)+\frac12)}{\Gamma(\frac{n}{2}(d-1))}\frac{(c^2t^2- ||\underline{\bf x}_d||^2)^{\frac{n}{2}(d-1)-1}}{\pi^{d/2}(ct)^{(n+1)(d-1)-1}},
\end{equation}
with $ d\geq 2$, $||\underline{\bf x}_d||<ct$ and $n\geq 1$. For $n=0$ the moving particle's position is uniformly distributed on the surface of the $d$-dimensional hypersphere $\mathcal{H}_{ct}^d$ with radius $ct$, and zero elsewhere. We remark that for $d=2$, we can extract from \eqref{eq:condlawint} the planar distribution
$$p_{\underline{\bf X}_2}(\underline{\bf x}_2,t;n)=\frac{n}{2\pi (ct)^n}(c^2t^2-||\underline{\bf x}_2||^2)^{\frac n2-1},\quad ||\underline{\bf x}_2||<ct,$$
see formula (1.1) in Orsingher and De Gregorio (2007). We also observe that \eqref{eq:condlawint} has the remarkable structure \eqref{eq:introd2} for $b=\frac n2(d-1)-1$ and the normalizing value equal to $$A=\frac{1}{\pi^{d/2}(ct)^{(n+1)(d-1)-1}}\frac{\Gamma(\frac{n+1}{2}(d-1)+\frac12)}{\Gamma(\frac{n}{2}(d-1))}$$
for all dimensions $d\geq 2$ and all numbers of changes of direction $n$.

Furthermore, if we assume \eqref{eq:jointdis3} as the joint distribution for the time length intervals $(\tau_1,...,\tau_n)$, the position $\underline{\bf Y}_d(t)=(Y_1(t),...,Y_d(t)),\,t>0,\,d\geq 3,$ of the random flight  has probability density which reads
\begin{equation}\label{eq:condlawint2}
 p_{\underline{\bf Y}_d}(\underline{\bf y}_d,t;n)
=\frac{\Gamma((n+1)(\frac d2-1)+1)}{\Gamma(n(\frac d2-1))}\frac{(c^2t^2- ||\underline{\bf y}_d||^2)^{n(\frac d2-1)-1}}{\pi^{d/2}(ct)^{2(n+1)(\frac d2-1)}},
\end{equation}
with $d\geq 3$, $||\underline{\bf y}_d||<ct$. For $d=4$, formula \eqref{eq:condlawint2} simplifies into
$$p_{\underline{\bf  Y}_4}(\underline{\bf y}_4,t;n)=\frac{n(n+1)}{\pi^2(ct)^{2n+2}}(c^2t^2- ||\underline{\bf y}_4||^2)^{n-1},$$
which coincides with the result (3.2) of Orsingher and De Gregorio (2007).

The technical reason for which this important and simple result is possible is due to the semigroup property of the Bessel functions \eqref{sem} (applicable with Dirichlet distribution \eqref{eq:jointdis2}) and \eqref{sem2} (applicable in the case of Dirichlet distribution \eqref{eq:jointdis3}). If we do not harmonize the order of the Dirichlet distribution with the dimension of space $d$, we get entangled in highly complicated formulae as Beghin and Orsingher (2010) showed. Throughout the paper the orientation (with distribution  \eqref{eq:jointdis1}) and the intertimes (with probability laws  \eqref{eq:jointdis2} and  \eqref{eq:jointdis3}) are assumed independent.

We also observe that \eqref{eq:condlawint2} corresponds to the uniform law if
$n=\frac{2}{d-2}$ 
and this implies that for $d=3$ we need two changes of direction in order to obtain the uniform law, while for $d=4$ one change of orientation leads to the same result. In the previous case \eqref{eq:condlawint} the condition to have a uniform distribution is
$n=\frac{2}{d-1}$
and this means that for $d=2$, $n$ must be equal one in order to obtain this strange and unexpected result.

For $n=\frac{1}{d-1}$ the law \eqref{eq:condlawint} reduces to the form
\begin{equation}\label{eq:lawinf}
p_{\underline{\bf X}_d}\left(\underline{\bf x}_d,t;\frac{1}{d-1}\right)=\frac{1}{\pi^{d/2}(ct)^{d-1}}\frac{\Gamma(\frac{d+1}{2})}{\Gamma(\frac{d}{2})}(c^2t^2- ||\underline{\bf x}_d||^2)^{-\frac{1}{2}},
\end{equation}
and this is valid for $d=2$. For $n=\frac{1}{d-2}$ the second law \eqref{eq:condlawint2} takes again the form \eqref{eq:lawinf} and this holds for $d=3$. Therefore for these two cases the distribution displays a singular behavior near the surface of the sphere. In general, the bigger the number $n$ of changes of direction the more concentrated around the starting point the distributions \eqref{eq:condlawint} and \eqref{eq:condlawint2} are. This is because the sample paths coil up around the origin since they are subject to contradictory, fragmented displacements. 

From \eqref{eq:condlawint} and \eqref{eq:condlawint2} we can extract the distributions of the random flights on all subspaces $\mathbb{R}^m,1\leq m<d,$ which preserve the structure \eqref{eq:introd2}. For $m=1$, we are able to obtain the distribution of the projection of the random flight on the line by means of order statistics (as previously elaborated for telegraph processes, see De Gregorio {\it et al.}, 2005, and some planar extensions with a finite number of directions, Leorato and Orsingher, 2004). The marginal distributions of \eqref{eq:condlawint} and \eqref{eq:condlawint2} can be interpreted as the probability law of the random flight described by the shadow on the subspace $\mathbb{R}^m,1\leq m<d$, or, equivalently, as a motion where changes of direction imply also a random change of the velocity in the subsequent displacements.

The projection of the distributions \eqref{eq:condlawint} and \eqref{eq:condlawint2} on the line yields
\begin{equation}\label{eq:projline}
f_{X_1}^d(x_1,t;n)=\frac{\Gamma((n+1)(d-1))}{(\Gamma(\frac{n+1}{2}(d-1)))^2}\frac{(c^2t^2-x_1^2)^{\frac{n+1}{2}(d-1)-1}}{(2ct)^{(n+1)(d-1)-1}}, 
\end{equation}
with $d\geq 2, |x_1|<ct,$ and
\begin{equation}\label{eq:projline2}
f_{Y_1}^d(y_1,t;n)=\frac{\Gamma(2(n+1)(\frac d2-1)+1)}{(\Gamma((n+1)(\frac d2-1)+\frac12))^2}\frac{(c^2t^2-y_1^2)^{(n+1)(\frac d2-1)-\frac12}}{(2ct)^{2(n+1)(\frac d2-1)}},
\end{equation}
with $ d\geq 3, |y_1|<ct.$ Particularly interesting is the case $d=3$ in \eqref{eq:projline}, because it yields the conditional distribution of the telegraph process $T(t),t>0$, that is
$$P\{T(t)\in dx|N(t)=2n+1\}=\frac{(2n+1)!}{(n!)^2}\frac{(c^2t^2-x^2)^n}{(2ct)^{2n+1}},\quad |x|<ct,$$
where $N(t),t>0,$ is a homogeneous Poisson process.
From \eqref{eq:projline2}, for $d=3$ and $n=2k-2,k\geq 1,$ we arrive instead at
$$P\{T(t)\in dy|N(t)=2k\}=\frac{(2k)!ct}{k!(k-1)!}\frac{(c^2t^2-y^2)^{k-1}}{(2ct)^{2k}},\quad |y|<ct,$$
which coincides with formula (2.18) of De Gregorio {\it et al.} (2005). For odd values of $d$ we can derive from \eqref{eq:projline} the distributions of reinforced alternating processes on the line described in De Gregorio {\it et al.} (2005).

The results \eqref{eq:condlawint} and \eqref{eq:condlawint2} depend on the dimension $d$ of the space $\mathbb{R}^d$ and on the number $n$ of changes of direction. In order to obtain unconditional distributions (as in $\mathbb{R}^2$ and $\mathbb{R}^4$ in earlier work see Stadje, 1987 and Orsingher and De Gregorio, 2007), we here assume that the number of changes of direction is randomized and has the structure of a fractional Poisson process independent from the Dirichlet r.v.'s representing the step lengths and independent also from the r.v.'s representing the orientation of each displacement. Therefore, we average distribution \eqref{eq:condlawint} by means of the following distribution of a fractional Poisson process $\mathcal{N}_d(t),t>0,$ 
\begin{equation}\label{int:lawgenpoi}
P\left\{\mathcal{N}_d(t)=n\right\}=\frac{1}{E_{\frac{d-1}{2},\frac d2}(\lambda t)}\frac{(\lambda t)^n}{\Gamma((\frac{d-1}{2})n+\frac d2)} ,
\end{equation}
with $\lambda >0,\,d\geq 2,\, n=0,1,...$,
while for \eqref{eq:condlawint2}, we must take the fractional Poisson process $\mathcal{M}_d(t),t>0,$ with probability distribution
\begin{equation*}
P\left\{\mathcal{M}_d(t)=n\right\}=\frac{1}{E_{\frac{d}{2}-1,\frac d2}(\lambda t)}\frac{(\lambda t)^n}{\Gamma((\frac{d}{2}-1)n+\frac d2)}, 
\end{equation*}
with $\lambda >0,d\geq 3,\,n=0,1,...$.
By combining \eqref{eq:condlawint} and \eqref{int:lawgenpoi} we obtain the probability law
\begin{equation*}
\frac{P\{\underline{\bf X}_{d}(t)\in d\underline{\bf x}_{d}\}}{\prod_{j=1}^ddx_j
}=\frac{\lambda t}{\pi^{\frac d2}}\frac{(c^2t^2-||\underline{\bf x}_{d}||^2)^{\frac{d-1}{2}-1}}{(ct)^{2(d-1)-1}}\frac{E_{\frac {d-1}{2},\frac {d-1}{2}}\left(\frac{\lambda t(c^2t^2-||\underline{\bf x}_{d}||^2)^{\frac {d-1}{2}}}{(ct)^{(d-1)}}\right)}{E_{\frac{d-1}{2},\frac d2}(\lambda t)}
\end{equation*}
which, for $d=3$, simplifies into
$$\frac{P\{\underline{\bf X}_{3}(t)\in d\underline{\bf x}_{3}\}}{\prod_{j=1}^3dx_j}=\frac{\lambda}{\pi^{\frac32}c^3t^2}\frac{e^{\frac{\lambda}{c^2t}(c^2t^2-||\underline{\bf x}_{3}||^2)}}{E_{1,\frac32}(\lambda t)},$$
which is similar to the unconditional distribution in $\mathbb{R}^4$ obtained in Orsingher and De Gregorio (2007). The functions appearing in the formulae above are called two-parameter Mittag-Leffler functions
$$E_{\nu,\beta}=\sum_{k=0}^\infty\frac{x^k}{\Gamma(\nu k+\beta)},\quad x\in\mathbb{R},\nu>0,\beta>0,$$
and play a central role in fractional calculus.

\section{Exact probability distributions for a random flight in $\mathbb{R}^d$}

We start our analysis by presenting the characteristic functions of the vector processes $\underline{\bf X}_d(t),t>0,$ and $\underline{\bf Y}_d(t),t>0,$ defined as \eqref{polar}, when the number of displacements is fixed and equal to $n+1,n\geq 1,$ and the changes of orientation are separated by random times $\tau_1,...,\tau_n$ with distribution \eqref{eq:jointdis2} and \eqref{eq:jointdis3} respectively. We denote by $\underline{\alpha}_d=(\alpha_1,...,\alpha_d)\in \mathbb{R}^d$ while $||\underline{{\bf x}}_d||=\sqrt{\sum_{j=1}^dx_j^2}$ and $<\underline{\alpha}_d,\underline{{\bf x}}_d>=\sum_{j=1}^d\alpha_jx_j$ represent the Euclidean distance and the scalar product, respectively.

\begin{theorem}\label{th1}
For the vector process $\underline{\bf X}_d(t)=(X_1(t),...,X_d(t)),t>0,$ with intermediate time lengths with joint distribution \eqref{eq:jointdis2}, the characteristic function reads 
\begin{equation}\label{cf}
E\left\{e^{i<\underline{\alpha}_d,\underline{\bf X}_d(t)>}\right\}=\frac{2^{\frac{n+1}{2}(d-1)-\frac12}\Gamma(\frac{n+1}{2}(d-1)+\frac12)}{(ct||\underline{\alpha}_d||)^{\frac{n+1}{2}(d-1)-\frac12}}J_{\frac{n+1}{2}(d-1)-\frac12}(ct||\underline{\alpha}_d||),
\end{equation}
where $d\geq 2$. For the vector process $\underline{\bf Y}_d(t)=(Y_1(t),...,Y_d(t)),t>0,$ with intertime lengths having joint distribution \eqref{eq:jointdis3}, the characteristic function is
\begin{equation}\label{cf2}
E\left\{e^{i<\underline{\alpha}_d,\underline{\bf Y}_d(t)>}\right\}=\frac{2^{(n+1)(\frac d2-1)}\Gamma((n+1)(\frac d2-1)+1)}{(ct||\underline{\alpha}_d||)^{(n+1)(\frac d2-1)}}J_{(n+1)(\frac d2-1)}(ct||\underline{\alpha}_d||),
\end{equation}
with $d\geq 3$, where
$$J_\nu(x)=\sum_{k=0}^\infty(-1)^k\left(\frac{x}{2}\right)^{2k+\nu}\frac{1}{k!\Gamma(k+\nu+1)},\quad x,\nu\in\mathbb{R},$$
is the Bessel function.
\end{theorem}

\begin{proof} We show that under the assumption that \eqref{eq:jointdis2} represents the joint distribution of the intervals $\tau_1,...,\tau_n$, the characteristic function of the position of the $d$-dimensional random motion $\underline{\bf X}_d(t)$ is equal to \eqref{cf}. We can write that
\begin{equation*}
E\left\{e^{i<\underline{\alpha}_d,\underline{\bf X}_d(t)>}\right\}=\int_0^{t}d\tau_1\int_0^{t-\tau_1}d\tau_2\cdots\int_0^{t-\sum_{j=1}^{n-1}\tau_j}d\tau_n\,f_1(\tau_1,...,\tau_n)\,\mathcal{I}_n(\underline{\alpha}_d)
\end{equation*}
where
\begin{align}\label{eq:I_n}
&\mathcal{I}_n(\underline{\alpha}_d)\\
&=\int_0^\pi d \theta_{1,1}\cdots\int_0^\pi d \theta_{1,n+1}\cdots \int_0^\pi d \theta_{d-2,1}\cdots\int_0^\pi d \theta_{d-2,n+1} \int_0^{2\pi}d \phi_{1}\cdots\int_0^{2\pi} d \phi_{n+1}\notag\\
&\quad\times \prod_{j=1}^{n+1}\Bigg\{\exp\{ic\tau_j(\alpha_d\sin\theta_{1,j}\sin\theta_{2,j}\cdot\cdot\cdot\sin\theta_{d-2,j}\sin\phi_{j}+\alpha_{d-1}\sin\theta_{1,j}\sin\theta_{2,j}\cdot\cdot\cdot\sin\theta_{d-2,j}\cos\phi_{j}\notag\\
&\quad+\cdots+\alpha_2\sin\theta_{1,j}\cos\theta_{2,j}+\alpha_1 \cos\theta_{1,j} ) \}\frac{\Gamma(d/2)}{2\pi^{d/2}}\sin\theta_{1,j}^{d-2}\cdot\cdot\cdot\sin\theta_{d-2,j}\Bigg\}\notag
\end{align}

The multiple integral \eqref{eq:I_n} is performed with respect to $(n+1)(d-2)$ angle variables $\theta_{i,j}$ and $n+1$ variables $\phi_j$, with $1\leq i\leq d-2,\,1\leq j\leq n+1$, appearing in the orientation distribution \eqref{eq:jointdis1}. The integral $\mathcal{I}_n(\underline{\alpha}_d)$ has been worked out by Orsingher and De Gregorio (2007) as follows. Since
$$
\frac{1}{2\pi}\int_0^{2\pi}e^{ix(a\cos\phi+b\sin\phi)}d\phi=J_0(x\sqrt{a^2+b^2}),
$$
we observe that, after integrations with respect to $\phi_j,j=1,...,n+1,$ \eqref{eq:I_n} becomes
\begin{align*}
&\mathcal{I}_n(\underline{\alpha}_d)\\
&=\int_0^\pi d \theta_{1,1}\cdots\int_0^\pi d \theta_{1,n+1}\cdots \int_0^\pi d \theta_{d-2,1}\cdots\int_0^\pi d \theta_{d-2,n+1} \\
&\quad\times \prod_{j=1}^{n+1}\Bigg\{\exp\{ic\tau_j(\alpha_{d-2}\sin\theta_{1,j}\sin\theta_{2,j}\cdot\cdot\cdot\sin\theta_{d-2,j}\cos\theta_{d-2,j} +\dots+\alpha_2\sin\theta_{1,j}\cos\theta_{2,j}+\alpha_1 \cos\theta_{1,j} ) \}\\
&\quad\times J_0\left(c\tau_j\sin\theta_{1,j}\cdot\cdot\cdot\sin\theta_{d-2,j}
\sqrt{\alpha_{d}^2+\alpha_{d-1}^2}\right)\frac{\Gamma(d/2)}{\pi^{d/2-1}}\sin\theta_{1,j}^{d-2}\cdot\cdot\cdot\sin\theta_{d-2,j}\Bigg\}
\end{align*}

We are able to perform all the $(d-2)(n+1)$ integrations with
respect to the angles $\theta_{i,j},1\leq i \leq d-2,j=1,...,n+1,$
by applying successively the formulas below
\begin{equation}\label{nine}
\int_0^{\pi/2}(\sin x)^{\nu+1}\cos(b\cos x)J_\nu(a\sin
x)dx=\sqrt{\frac{\pi}{2}}\frac{a^\nu J_{\nu+\frac{1}{2}}\left(\sqrt{a^2+b^2}\right)}{(a^2+b^2)^{\frac{\nu}{2}+\frac{1}{4}}}
,
\end{equation}
for $Re~\nu>-1$ (see Gradshteyn-Ryzhik, 1980, pag. 743, formula
6.688.(2)). The integration with respect to
$\theta_{d-2,j},~j=1,...,n+1$ yields
\begin{equation*}\label{eleven}
\begin{split}
&\int_0^\pi d\theta_{d-2,1}\cdot\cdot\cdot \int_0^\pi
d\theta_{d-2,n+1}\prod_{j=1}^{n+1}e^{ic \tau_j\alpha_{d-2}
\sin\theta_{1,j}\cdot\cdot\cdot\sin\theta_{d-3,j}\cos\theta_{d-2,j}}\sin\theta_{d-2,j}\\
&J_0\left(c\tau_j\sin\theta_{1,j}\cdot\cdot\cdot\sin\theta_{d-2,j}
\sqrt{\alpha_{d}^2+\alpha_{d-1}^2}\right)
\\
&=\prod_{j=1}^{n+1}\Bigg\{\int_0^\pi e^{ic
\tau_j\alpha_{d-2}
\sin\theta_{1,j}\cdot\cdot\cdot\sin\theta_{d-3,j}\cos\theta_{d-2,j}}\sin\theta_{d-2,j}J_0\left(c\tau_j\sin\theta_{1,j}\cdot\cdot\cdot\sin\theta_{d-2,j}
\sqrt{\alpha_{d}^2+\alpha_{d-1}^2}\right)d\theta_{d-2,j}\Bigg\}\\
&=\prod_{j=1}^{n+1}\Bigg\{2\int_0^{\pi/2}\cos (c
\tau_j\alpha_{d-2}
\sin\theta_{1,j}\cdot\cdot\cdot\sin\theta_{d-3,j}\cos\theta_{d-2,j})\\
&\quad\sin\theta_{d-2,j}J_0\left(c\tau_j\sin\theta_{1,j}\cdot\cdot\cdot\sin\theta_{d-2,j}
\sqrt{\alpha_{d}^2+\alpha_{d-1}^2}\right)d\theta_{d-2,j}\Bigg\}\\
&=\prod_{j=1}^{n+1}\left\{2\sqrt{\frac{\pi}{2}}\frac{J_{1/2}\left(c\tau_j\sin\theta_{1,j}\cdot\cdot\cdot\sin\theta_{d-3,j}\sqrt{\alpha_{d}^2+\alpha_{d-1}^2
+\alpha_{d-2}^2}\right)}{\left(c\tau_j\sin\theta_{1,j}\cdot\cdot\cdot\sin\theta_{d-3,j}\sqrt{\alpha_{d}^2+\alpha_{d-1}^2
+\alpha_{d-2}^2}\right)^{1/2}}\right\}.
\end{split}
\end{equation*}

In the last step we applied formula \eqref{nine} for
\[
\nu=0,\, a=c\tau_j\sin\theta_{1,j}\cdot\cdot\sin\theta_{d-3,j}\sqrt{\alpha_{d}^2+\alpha_{d-1}^2
},\, b=c\tau_j\alpha_{d-2}\sin\theta_{1,j}\cdot\cdot\sin\theta_{d-3,j}
\]
and also considered that
\[
\int_0^\pi \sin(\beta\cos x )(\sin x)^{\nu+1}J_\nu(\alpha\sin x)dx=0.
\]

 The integration with respect to
the variables $\theta_{d-3,1},...,\theta_{d-3,n+1}$ follows
similarly by applying again \eqref{nine} and yields
\begin{equation*}\label{twelve}
\begin{split}
&\int_0^\pi d\theta_{d-3,1}\cdot\cdot\cdot \int_0^\pi
d\theta_{d-3,n+1}\prod_{j=1}^{n+1}e^{ic \tau_j\alpha_{d-3}
\sin\theta_{1,j}\cdot\cdot\sin\theta_{d-4,j}\cos\theta_{d-3,j}}\sin\theta^2_{d-3,j}\\
&\times2\sqrt{\frac{\pi}{2}}\frac{J_{1/2}\left(c\tau_j\sin\theta_{1,j}\cdot\cdot\sin\theta_{d-3,j}\sqrt{\alpha_{d}^2+\alpha_{d-1}^2
+\alpha_{d-2}^2}\right)}{\left(c\tau_j\sin\theta_{1,j}\cdot\cdot\sin\theta_{d-3,j}\sqrt{\alpha_{d}^2+\alpha_{d-1}^2
+\alpha_{d-2}^2}\right)^{1/2}}\\
&=\prod_{j=1}^{n+1}\left(2\sqrt{\frac{\pi}{2}}\right)^2\frac{J_1\left(c\tau_j\sin\theta_{1,j}\cdot\cdot\sin\theta_{d-4,j}\sqrt{\alpha_{d}^2+\alpha_{d-1}^2
+\alpha_{d-2}^2+\alpha_{d-3}^2}\right)}{\left(c\tau_j\sin\theta_{1,j}\cdot\cdot\sin\theta_{d-4,j}\sqrt{\alpha_{d}^2+\alpha_{d-1}^2
+\alpha_{d-2}^2+\alpha_{d-3}^2}\right)}.
\end{split}
\end{equation*}

By continuing in the same way, that is by applying successively
formula \eqref{nine} we obtain that
\begin{equation}\label{intangle}
\mathcal{I}_n(\underline{\alpha}_d)=\left\{2^{\frac d2-1}\Gamma\left(\frac d2\right)\right\}^{n+1}\prod_{j=1}^{n+1}\frac{J_{\frac d2-1}(c\tau_j||\underline{\alpha}_d||)}{(c\tau_j||\underline{\alpha}_d||)^{\frac d2-1}}
\end{equation}

Therefore, the characteristic function becomes
\begin{align}\label{eq:cf}
E\left\{e^{i<\underline{\alpha}_d,\underline{\bf X}_d(t)>}\right\}
&=\left\{2^{\frac d2-1}\Gamma\left(\frac d2\right)\right\}^{n+1}\int_0^{t}d\tau_1\int_0^{t-\tau_1}d\tau_2\cdots\int_0^{t-\sum_{j=1}^{n-1}\tau_j}d\tau_n\,f_1(\tau_1,...,\tau_n)\notag\\
&\quad\times\prod_{j=1}^{n+1}\frac{J_{\frac d2-1}(c\tau_j||\underline{\alpha}_d||)}{(c\tau_j||\underline{\alpha}_d||)^{\frac d2-1}}\notag\\
&=\left\{2^{\frac d2-1}\Gamma\left(\frac d2\right)\right\}^{n+1}\frac{\Gamma((n+1)(d-1))}{(\Gamma(d-1))^{n+1}}\frac{1}{t^{(n+1)(d-1)-1}}\notag\\
&\quad\times\int_0^{t}\tau_1^{d-2}d\tau_1\int_0^{t-\tau_1}\tau_2^{d-2}d\tau_2\cdots\int_0^{t-\sum_{j=1}^{n-1}\tau_j}\tau_{n}^{d-2}(t-\sum_{j=1}^{n}\tau_j)^{d-2}d\tau_n\notag\\
&\quad\times\prod_{j=1}^{n+1}\frac{J_{\frac d2-1}(c\tau_j||\underline{\alpha}_d||)}{(c\tau_j||\underline{\alpha}_d||)^{\frac d2-1}}
\end{align}
In order to work out this $n$-fold integral, the following result (see Gradshteyn-Ryzhik, 1980, pag. 743, formula
6.581(3)) 
\begin{equation}\label{sem}
\int_0^ax^\mu(a-x)^\nu J_\mu(x)J_\nu(a-x)dx=\frac{\Gamma(\mu+\frac12)\Gamma(\nu+\frac12)}{\sqrt{2\pi}\Gamma(\mu+\nu+1)}a^{\mu+\nu+\frac12}J_{\mu+\nu+\frac12}(a),
\end{equation}
with $Re\,\mu>-\frac12$ and $Re\,\nu>-\frac12$, assumes a crucial role. Indeed, we apply recursively the formula \eqref{sem} to calculate each integral with respect to the variable $\tau_j$. In the first step we have therefore
\begin{align*}
&\int_0^{t-\sum_{j=1}^{n-1}\tau_j}\tau_{n}^{d-2}(t-\sum_{j=1}^{n}\tau_j)^{d-2}\frac{J_{\frac d2-1}(c\tau_n||\underline{\alpha}_d||)}{(c\tau_n||\underline{\alpha}_d||)^{\frac d2-1}}\frac{J_{\frac d2-1}(c(t-\sum_{j=1}^{n}\tau_j)||\underline{\alpha}_d||)}{(c(t-\sum_{j=1}^{n}\tau_j)||\underline{\alpha}_d||)^{\frac d2-1}}d\tau_n\\
&=\int_0^{t-\sum_{j=1}^{n-1}\tau_j}\frac{d \tau_n}{(c||\underline{\alpha}_d||)^{2d-4}}(c\tau_n||\underline{\alpha}_d||)^{\frac d2-1}(c(t-\sum_{j=1}^{n}\tau_j)||\underline{\alpha}_d||)^{\frac d2-1}\\
&\quad \times J_{\frac d2-1}(c\tau_n||\underline{\alpha}_d||)J_{\frac d2-1}(c(t-\sum_{j=1}^{n}\tau_j)||\underline{\alpha}_d||)=(c\tau_n||\underline{\alpha}_d||=y )
\end{align*}
\begin{align*}
&=\frac{1}{(c||\underline{\alpha}_d||)^{2d-3}}\int_0^{c(t-\sum_{j=1}^{n-1}\tau_j)||\underline{\alpha}_d||}dyy^{\frac d2-1}(c(t-\sum_{j=1}^{n-1}\tau_j)||\underline{\alpha}_d||-y)^{\frac d2-1}\\
&\quad\times J_{\frac d2-1}(y)J_{\frac d2-1}(c(t-\sum_{j=1}^{n-1}\tau_j)||\underline{\alpha}_d||-y)\\
&=\frac{1}{(c||\underline{\alpha}_d||)^{2d-3}}\frac{\left(\Gamma\left(\frac{d-1}{2}\right)\right)^2}{\sqrt{2\pi}\Gamma(d-1)}(c(t-\sum_{j=1}^{n-1}\tau_j)||\underline{\alpha}_d||)^{d-\frac32}J_{d-\frac 32}(c(t-\sum_{j=1}^{n-1}\tau_j)||\underline{\alpha}_d||)
\end{align*}

The second integral is given by
\begin{eqnarray*}
&&\frac{\left(\Gamma\left(\frac{d-1}{2}\right)\right)^2}{\sqrt{2\pi}\Gamma(d-1)}\frac{1}{(c||\underline{\alpha}_d||)^{3d-5}}
\int_0^{t-\sum_{j=1}^{n-2}\tau_j}d\tau_{n-1}\\
&&\times(c\tau_{n-1}||\underline{\alpha}_d||)^{\frac d2-1}(c(t-\sum_{j=1}^{n-1}\tau_j)||\underline{\alpha}_d||)^{d-\frac32}J_{\frac d2-1}(c\tau_{n-1}||\underline{\alpha}_d||)J_{d-\frac 32}(c(t-\sum_{j=1}^{n-1}\tau_j)||\underline{\alpha}_d||)=(c\tau_n||\underline{\alpha}_d||=y )\\
&&=\frac{\left(\Gamma\left(\frac{d-1}{2}\right)\right)^2}{\sqrt{2\pi}\Gamma(d-1)}\frac{1}{(c||\underline{\alpha}_d||)^{3d-4}}
\int_0^{c(t-\sum_{j=1}^{n-2}\tau_j)||\underline{\alpha}_d||}dy\\
&&\quad\times y^{\frac d2-1}(c(t-\sum_{j=1}^{n-2}\tau_j)||\underline{\alpha}_d||-y)^{d-\frac32}J_{\frac d2-1}(y)J_{d-\frac 32}(c(t-\sum_{j=1}^{n-2}\tau_j)||\underline{\alpha}_d||-y)\\
&&=\frac{\left(\Gamma\left(\frac{d-1}{2}\right)\right)^3}{(\sqrt{2\pi})^2\Gamma(\frac32(d-1))}\frac{1}{(c||\underline{\alpha}_d||)^{3d-4}}(c(t-\sum_{j=1}^{n-2}\tau_j)||\underline{\alpha}_d||)^{\frac32d-2}J_{\frac32d-2}(c(t-\sum_{j=1}^{n-2}\tau_j)||\underline{\alpha}_d||)
\end{eqnarray*}
By considering formula \eqref{sem}, we see that after $(n-1)$ integrations (with $n\geq 2$), the exponent of $c(t-\tau_1)||\underline{\alpha}_d||$ as well as the order of the Bessel function is given by the formula
$$\left(\frac d2-1\right)+\left[\left(\frac d2-1\right)(n-1)+\frac12(n-2)\right]+\frac12=\frac n2(d-1)-\frac12$$
with $\mu=\frac d2-1$, $\nu=\left(\frac d2-1\right)(n-1)+\frac12(n-2)$. The exponent of $c||\underline{\alpha}_d||$ can be obtained by adding to $dn-(n+1)$ (where $n$ here is the number of integrations) the number $2(\frac d2-1)$ due to the adjustements necessary to apply \eqref{sem}.

Then, the last integral becomes
\begin{align}\label{eq:lastint}
&\frac{\left(\Gamma\left(\frac{d-1}{2}\right)\right)^n}{(\sqrt{2\pi})^{n-1}\Gamma(\frac n2(d-1))}\frac{1}{(c||\underline{\alpha}_d||)^{(d-1)(n+1)-2}}
\int_0^{t}d\tau_{1}\\
&\times(c\tau_{1}||\underline{\alpha}_d||)^{\frac d2-1}(c(t-\tau_1)||\underline{\alpha}_d||)^{\frac n2(d-1)-\frac12}J_{\frac d2-1}(c\tau_{1}||\underline{\alpha}_d||)J_{\frac n2(d-1)-\frac12}(c(t-\tau_1)||\underline{\alpha}_d||)=(c\tau_1||\underline{\alpha}_d||=y )\notag\\
&=\frac{\left(\Gamma\left(\frac{d-1}{2}\right)\right)^n}{(\sqrt{2\pi})^{n-1}\Gamma(\frac n2(d-1))}\frac{1}{(c||\underline{\alpha}_d||)^{(d-1)(n+1)-1}}
\int_0^{ct||\underline{\alpha}_d||}dy\notag\\
&\quad\times y^{\frac d2-1}(ct||\underline{\alpha}_d||-y)^{\frac n2(d-1)-\frac12}J_{\frac d2-1}(y)J_{\frac n2(d-1)-\frac12}(ct||\underline{\alpha}_d||-y)\notag\\
&=\frac{\left(\Gamma\left(\frac{d-1}{2}\right)\right)^{n+1}}{(\sqrt{2\pi})^n\Gamma\left(\frac{(d-1)}{2}(n+1)\right)}\frac{(ct||\underline{\alpha}_d||)^{\frac{n+1}2(d-1)-\frac12}}{(c||\underline{\alpha}_d||)^{(d-1)(n+1)-1}}J_{\frac{n+1}2(d-1)-\frac12}(ct||\underline{\alpha}_d||)\notag
\end{align}
Therefore, plugging the result \eqref{eq:lastint} into the expression \eqref{eq:cf}, and by observing that by means of the duplication formula we have that
$$\Gamma\left(\frac d2\right)=\sqrt{\pi}2^{2-d}\frac{\Gamma(d-1)}{\Gamma(\frac{d-1}{2})}$$
and 
$$\Gamma\left(\frac{(n+1)}{2}(d-1)+\frac12\right)=\sqrt{\pi}2^{1-(n+1)(d-1)}\frac{\Gamma((n+1)(d-1))}{\Gamma(\frac{(n+1)}{2}(d-1))},$$
some simplifications lead to result \eqref{cf}.

Under the assumption that the $f_2(\tau_1,...,\tau_n)$ is the density law for the intertimes $\tau_j,j=1,...,n+1,$ and by using arguments similar to those of the first part of the proof, the characteristic function of $\underline{\bf Y}_d(t),t>0,$ assumes the following integral form
\begin{align*}\label{eq:cf2}
E\left\{e^{i<\underline{\alpha}_d,\underline{\bf Y}_d(t)>}\right\}
&=\left\{2^{\frac d2-1}\Gamma\left(\frac d2\right)\right\}^{n+1}\frac{\Gamma((n+1)(\frac d2-1))}{(\Gamma(\frac d2-1))^{n+1}}\frac{1}{t^{(n+1)(\frac d2-1)-1}}\notag\\
&\quad\times\int_0^{t}\tau_1^{\frac d2-2}d\tau_1\int_0^{t-\tau_1}\tau_2^{\frac d2-2}d\tau_2\cdots\int_0^{t-\sum_{j=1}^{n-1}\tau_j}\tau_{n}^{\frac d2-2}(t-\sum_{j=1}^{n}\tau_j)^{\frac d2-2}d\tau_n\prod_{j=1}^{n+1}\frac{J_{\frac d2-1}(c\tau_j||\underline{\alpha}_d||)}{(c\tau_j||\underline{\alpha}_d||)^{\frac d2-1}}\notag
\end{align*}
The first integral with respect to $\tau_n$ becomes
\begin{align*}
&\frac{1}{(c||\underline{\alpha}_d||)^{d-4}}\int_0^{t-\sum_{j=1}^{n-1}\tau_j}\frac{J_{\frac d2-1}(c\tau_n||\underline{\alpha}_d||)}{c\tau_n||\underline{\alpha}_d||}\frac{J_{\frac d2-1}(c(t-\sum_{j=1}^n\tau_j)||\underline{\alpha}_d||)}{c(t-\sum_{j=1}^n\tau_j)||\underline{\alpha}_d||}d\tau_n=(y=c\tau_n||\underline{\alpha}_d||)\\
&=\frac{1}{(c||\underline{\alpha}_d||)^{d-3}}\int_0^{c(t-\sum_{j=1}^{n-1}\tau_j)||\underline{\alpha}_d||}\frac{J_{\frac d2-1}(y)}{y}\frac{J_{\frac d2-1}(c(t-\sum_{j=1}^{n-1}\tau_j)||\underline{\alpha}_d||-y)}{c(t-\sum_{j=1}^{n-1}\tau_j)||\underline{\alpha}_d||-y}dy\\
&=\frac{1}{(c||\underline{\alpha}_d||)^{d-3}}\frac{2}{\frac d2-1}\frac{J_{2\left(\frac d2-1\right)}(c(t-\sum_{j=1}^{n-1}\tau_j)||\underline{\alpha}_d||)}{c(t-\sum_{j=1}^{n-1}\tau_j)||\underline{\alpha}_d||}
\end{align*}
where in the last step we have used the following formula (see Gradshteyn-Ryzhik, 1980, pag. 678, formula
6.533.(2))
\begin{equation}\label{sem2}
\int_0^a\frac{J_\mu(x)J_\nu(a-x)}{x(a-x)}dx=\left(\frac1\mu+\frac1\nu\right)\frac{J_{\mu+\nu}(a)}{a},\quad Re\, \mu>0,\, Re\,\nu>0.
\end{equation}
The second integral provides us
\begin{align*}
&\frac{1}{(c||\underline{\alpha}_d||)^{\frac32 d-5}}\frac{2}{\frac d2-1}\int_0^{t-\sum_{j=1}^{n-2}\tau_j}\frac{J_{\frac d2-1}(c\tau_{n-1}||\underline{\alpha}_d||)}{c\tau_{n-1}||\underline{\alpha}_d||}\frac{J_{2\left(\frac d2-1\right)}(c(t-\sum_{j=1}^{n-1}\tau_j)||\underline{\alpha}_d||)}{c(t-\sum_{j=1}^{n-1}\tau_j)||\underline{\alpha}_d||}d\tau_{n-1}=(y=c\tau_n||\underline{\alpha}_d||)\\
&=\frac{1}{(c||\underline{\alpha}_d||)^{\frac32 d-4}}\frac{2}{\frac d2-1}\int_0^{c(t-\sum_{j=1}^{n-2}\tau_j)||\underline{\alpha}_d||}\frac{J_{\frac d2-1}(y)}{y}\frac{J_{2\left(\frac d2-1\right)}(c(t-\sum_{j=1}^{n-2}\tau_j)||\underline{\alpha}_d||-y)}{c(t-\sum_{j=1}^{n-2}\tau_j)||\underline{\alpha}_d||-y}dy\\
&=\frac{1}{(c||\underline{\alpha}_d||)^{\frac32 d-4}}\frac{3}{(\frac d2-1)^2}\frac{J_{3\left(\frac d2-1\right)}(c(t-\sum_{j=1}^{n-2}\tau_j)||\underline{\alpha}_d||)}{c(t-\sum_{j=1}^{n-2}\tau_j)||\underline{\alpha}_d||}
\end{align*}

In the last integral, the exponent of $c||\underline{\alpha}_d||$ is equal to $(n-1+2)(\frac d2-1)-2$ ($n-1$, with $n\geq 2$, is the number of integrations performed). Therefore

\begin{align*}
&\frac{1}{(c||\underline{\alpha}_d||)^{(n+1)(\frac d2 -1)-2}}\frac{n}{(\frac d2-1)^{n-1}}\int_0^{t}\frac{J_{\frac d2-1}(c\tau_{1}||\underline{\alpha}_d||)}{c\tau_1||\underline{\alpha}_d||}\frac{J_{n\left(\frac d2-1\right)}(c(t-\tau_1)||\underline{\alpha}_d||)}{c(t-\tau_1)||\underline{\alpha}_d||}d\tau_{1}=(y=c\tau_1||\underline{\alpha}_d||)\\
&=\frac{1}{(c||\underline{\alpha}_d||)^{(n+1)(\frac d2 -1)-1}}\frac{n}{(\frac d2-1)^{n-1}}\int_0^{ct||\underline{\alpha}_d||}\frac{J_{\frac d2-1}(y)}{y}\frac{J_{n\left(\frac d2-1\right)}(ct||\underline{\alpha}_d||-y)}{ct||\underline{\alpha}_d||-y}dy\\
&=\frac{1}{(c||\underline{\alpha}_d||)^{(n+1)(\frac d2 -1)-1}}\frac{n+1}{(\frac d2-1)^{n}}\frac{J_{(n+1)\left(\frac d2-1\right)}(ct||\underline{\alpha}_d||)}{ct||\underline{\alpha}_d||}
\end{align*}
and then the result \eqref{cf2} follows immediately.

\end{proof}

In the next Theorem we are able to invert the characteristic functions \eqref{cf} and \eqref{cf2}.
\begin{theorem}\label{th2}
The probability laws of $\underline{\bf X}_d(t),t>0,$ and $\underline{\bf Y}_d(t),t>0,$ are respectively equal to
\begin{equation}\label{condlaw}
p_{\underline{\bf X}_d}(\underline{\bf x}_d,t;n)=\frac{\Gamma(\frac{n+1}{2}(d-1)+\frac12)}{\Gamma(\frac{n}{2}(d-1))}\frac{(c^2t^2- ||\underline{\bf x}_d||^2)^{\frac{n}{2}(d-1)-1}}{\pi^{d/2}(ct)^{(n+1)(d-1)-1}},
\end{equation}
$d\geq 2$, and
\begin{equation}\label{condlaw2}
p_{\underline{\bf Y}_d}(\underline{\bf y}_d,t;n)=\frac{\Gamma((n+1)(\frac d2-1)+1)}{\Gamma(n(\frac d2-1))}\frac{(c^2t^2- ||\underline{\bf y}_d||^2)^{n(\frac d2-1)-1}}{\pi^{d/2}(ct)^{2(n+1)(\frac d2-1)}},
\end{equation}
$d\geq 3$, with $||\underline{\bf x}_d||<ct$, $||\underline{\bf y}_d||<ct$ and $n\geq 1$.
\end{theorem}
\begin{proof}
By inverting the characteristic function \eqref{cf}, we are able to show that the 
density law of the process $\underline{\bf X}_d(t),t>0,$ is given by \eqref{condlaw}. Therefore, by passing to the hyperspherical coordinates, we have that
\begin{align*}
p_{\underline{\bf X}_d}(\underline{\bf x}_d,t;n)
&=\frac{1}{(2\pi)^d}\int_{\mathbb{R}^d}e^{-i<\underline{\alpha}_d,\underline{\bf x}_d>}E\left\{e^{i<\underline{\alpha}_d,\underline{\bf X}_d(t)>}\right\}d\alpha_1\cdots d\alpha_d\notag\\
&=\frac{1}{(2\pi)^d}\int_0^\infty \rho^{d-1}d\rho\int_0^\pi d\theta_1\cdots\int_0^\pi d\theta_{d-2}\int_0^{2\pi} d	\phi \sin^{d-2}\theta_1\cdots\sin\theta_{d-2}\notag\\
&\times \exp\{-i\rho(x_d\sin\theta_1\cdots\sin\theta_{d-2}\sin\phi+\cdots+x_2\sin\theta_1\cos\theta_2+x_1\cos\theta_1)\notag\\
&
\times\notag\frac{2^{\frac{n+1}{2}(d-1)-\frac12}\Gamma(\frac{n+1}{2}(d-1)+\frac12)}{(ct\rho)^{\frac{n+1}{2}(d-1)-\frac12}}J_{\frac{n+1}{2}(d-1)-\frac12}(ct\rho)\notag\\
&=\frac{2^{\frac{n+1}{2}(d-1)-\frac12}}{(2\pi)^{d/2}}\Gamma\left(\frac{n+1}{2}(d-1)+\frac12\right)\int_0^\infty\rho^{d-1}\frac{J_{\frac d2-1}(\rho ||\underline{\bf x}_d||)}{(\rho||\underline{\bf x}_d||)^{\frac d2-1}}\frac{J_{\frac{n+1}{2}(d-1)-\frac12}(ct\rho)}{(ct\rho)^{\frac{n+1}{2}(d-1)-\frac12}}d	\rho\notag\\
&=\frac{2^{\frac{n+1}{2}(d-1)-\frac12}}{(2\pi)^{d/2}}\frac{\Gamma\left(\frac{n+1}{2}(d-1)+\frac12\right)}{(ct)^{\frac{n+1}{2}(d-1)-\frac12}||\underline{\bf x}_d||^{\frac d2-1}}\int_0^\infty \rho^{\frac {d+1}{2}-\frac{n+1}{2}(d-1)}J_{\frac d2-1}(\rho ||\underline{\bf x}_d||)J_{\frac{n+1}{2}(d-1)-\frac12}(ct\rho)d\rho\notag\\
&=\frac{1}{\pi^{d/2}(ct)^{(n+1)(d-1)-1}}\frac{\Gamma(\frac{n+1}{2}(d-1)+\frac12)}{\Gamma(\frac{n+1}{2}(d-1)-\frac d2+\frac12)}(c^2t^2- ||\underline{\bf x}_d||^2)^{\frac{n+1}{2}(d-1)-\frac {d+1}{2}}
\end{align*}
In the first step above we have performed calculations similar to those leading to \eqref{intangle} and then
\begin{align}\label{intangle2}
&\int_0^\pi d\theta_1\cdots\int_0^\pi d\theta_{d-2}\int_0^{2\pi} d\phi \sin^{d-2}\theta_1\cdots\sin\theta_{d-2}\\
&\exp{\{-i\rho(x_d\sin\theta_1\cdots\sin\theta_{d-2}\sin\phi+\cdots+x_2\sin\theta_1\cos\theta_2+x_1\cos\theta_1)\}}\notag\\
&=(2\pi)^{\frac d2}\frac{J_{\frac d2-1}(\rho ||\underline{\bf x}_d||)}{(\rho ||\underline{\bf x}_d||)^{\frac d2-1}}\notag,
\end{align}
 while in the last step  we have used the formula (see Gradshteyn-Ryzhik, 1980, pag. 692, formula
6.575.(1) with a correction in the bounds of $\mu$ and $\nu$)
$$\int_0^\infty J_{\nu+1}(\alpha x)J_\mu(\beta x)x^{\mu-\nu}dx=\frac{(\alpha^2-\beta^2)^{\nu-\mu}\beta^\mu}{2^{\nu-\mu}\alpha^{\nu+1}\Gamma(\nu-\mu+1)},$$
$ \alpha\geq\beta,\,Re\,(\nu+1)>Re(\mu)>0$,
for $\nu=\frac{n+1}{2}(d-1)-\frac32$, $\mu=\frac d2-1$, $\alpha=ct$ and $\beta= ||\underline{\bf x}_d||$.

Analogously, for the random flight ${\bf Y}_d(t),t>0$, we have that
\begin{align*}
p_{\underline{\bf Y}_d}(\underline{\bf y}_d,t;n)
 &=\frac{2^{(n+1)(\frac d2-1)}\Gamma((n+1)(\frac d2-1)+1)}{(2\pi)^{\frac d2}}\int_0^\infty\rho^{d-1}\frac{J_{\frac d2-1}(\rho ||\underline{\bf x}_d||)}{(\rho||\underline{\bf x}_d||)^{\frac d2-1}}\frac{J_{(n+1)(\frac d2-1)}(ct\rho)}{(ct\rho)^{(n+1)(\frac d2-1)}}d	\rho\notag\\
&=\frac{2^{(n+1)(\frac d2-1)}\Gamma((n+1)(\frac d2-1)+1)}{(2\pi)^{\frac d2}(ct)^{(n+1)(\frac d2-1)}||\underline{\bf x}_d||^{\frac d2-1}}\int_0^\infty \rho^{\frac {d}{2}-(n+1)(\frac d2-1)}J_{\frac d2-1}(\rho ||\underline{\bf x}_d||)J_{(n+1)(\frac d2-1)}(ct\rho)d\rho\notag\\
&=\frac{1}{\pi^{d/2}}\frac{\Gamma((n+1)(\frac d2-1)+1)}{(ct)^{2(n+1)(\frac d2-1)}\Gamma(n(\frac d2-1))}(c^2t^2- ||\underline{\bf x}_d||^2)^{n(\frac d2-1)-1}
\end{align*}
\end{proof}

\begin{remark}
We can also check results \eqref{condlaw} and \eqref{condlaw2} by evaluating their Fourier transforms, thus showing that we reobtain results \eqref{cf} and \eqref{cf2} respectively. For $\mathcal{H}_{ct}^d=\{\underline{\bf x}_d:||\underline{\bf x}_d||<ct\}$, we have that

\begin{align}\label{eq:proofcf}
&\int_{\mathcal{H}_{ct}^d}e^{i<\underline{\alpha}_d,\underline{\bf x}_d>}p_{\underline{\bf  X}_d}(\underline{\bf x}_d,t;n) dx_1\cdots dx_d\\
&=\frac{\Gamma(\frac{n+1}{2}(d-1)+\frac12)}{\pi^{d/2}\Gamma(\frac{n}{2}(d-1))(ct)^{(n+1)(d-1)-1}}\int_{\mathcal{H}_{ct}^d}e^{i<\underline{\alpha}_d,\underline{\bf x}_d(t)>}(c^2t^2- ||\underline{\bf x}_d||^2)^{\frac{n}{2}(d-1)-1}dx_1\cdots dx_d\notag\\
&=\frac{\Gamma(\frac{n+1}{2}(d-1)+\frac12)}{\pi^{d/2}(ct)^{(n+1)(d-1)-1}\Gamma(\frac{n}{2}(d-1))}\int_0^{ct} \rho^{d-1}d\rho\int_0^\pi d\theta_1\cdots\int_0^\pi d\theta_{d-2}\int_0^{2\pi} d	\phi\notag\\
&\quad\times \exp\{i\rho(\alpha_d\sin\theta_1\cdots\sin\theta_{d-2}\sin\phi+\cdots+\alpha_2\sin\theta_1\cos\theta_2+\alpha_1\cos\theta_1) \notag\\
&\quad\times (c^2t^2- \rho^2)^{\frac{n}{2}(d-1)-1} \sin^{d-2}\theta_1\cdots\sin\theta_{d-2}\notag\\
&=\frac{2^{d/2}\Gamma(\frac{n+1}{2}(d-1)+\frac12)}{(ct)^{(n+1)(d-1)-1}\Gamma(\frac{n}{2}(d-1))}\int_0^{ct} \rho^{d-1} (c^2t^2- \rho^2)^{\frac{n}{2}(d-1)-1}\frac{J_{\frac d2-1}(\rho||\underline{\alpha}_d||)}{(\rho||\underline{\alpha}_d||)^{\frac d2-1}}d\rho,\notag
\end{align}
where in the last step we have used the result \eqref{intangle2}.

Now, we work out the previous integral
\begin{align}\label{eq:proofcf2}
&\int_0^{ct} \rho^{d-1} (c^2t^2- \rho^2)^{\frac{n}{2}(d-1)-1}\frac{J_{\frac d2-1}(\rho||\underline{\alpha}_d||)}{(\rho||\underline{\alpha}_d||)^{\frac d2-1}}d\rho\\
&=\sum_{k=0}^\infty\frac{(-1)^k}{k!\Gamma(k+\frac d2)}\frac{(||\underline{\alpha}_d||)^{2k}}{2^{2k+\frac d2-1}}\int_0^{ct}\rho^{d+2k-1}(c^2t^2-\rho^2)^{\frac{n}{2}(d-1)-1}d\rho=(\rho=ct\sqrt{y})\notag\\		
&=\sum_{k=0}^\infty\frac{(-1)^k}{k!\Gamma(k+\frac d2)}\frac{(||\underline{\alpha}_d||)^{2k}}{2^{2k+\frac d2}}(ct)^{(n+1)(d-1)+2k-1}\int_0^{1}y^{\frac d2+k-1}(1-y)^{\frac{n}{2}(d-1)-1}dy\notag\\
&=	\frac{(ct)^{(n+1)(d-1)-1}}{2^{\frac d2}}\Gamma\left(\frac{n}{2}(d-1)\right)\sum_{k=0}^\infty\frac{(-1)^k}{k!}\frac{(ct||\underline{\alpha}_d||)^{2k}}{2^{2k}}\frac{1}{\Gamma(k+\frac{n+1}{2}(d-1)+\frac12)}\notag
\end{align}
Therefore, plugging in the expression \eqref{eq:proofcf2} in \eqref{eq:proofcf}, we immediately obtain the characteristic function \eqref{cf}.

Similar calculations hold for the characteristic function of the distribution \eqref{condlaw2} and then it is not hard to obtain the result \eqref{cf2}.
\end{remark}

\begin{remark}
It is not hard to show that $p_{\underline{\bf X}_d}(\underline{\bf x}_d,t;n)$ and $p_{\underline{\bf Y}_d}(\underline{\bf y}_d,t;n)$ integrate to 1. Indeed
\begin{align*}
\int_{\mathcal{H}_{ct}^d}p_{\underline{\bf X}_d}(\underline{\bf x}_d,t;n)\prod_{j=1}^ddx_j
&=\frac{2\Gamma(\frac{n+1}{2}(d-1)+\frac12)}{\Gamma(\frac d2)(ct)^{(n+1)(d-1)-1}\Gamma(\frac{n}{2}(d-1))}\int_0^{ct}\rho^{d-1}(c^2t^2- \rho^2)^{\frac{n}{2}(d-1)-1}d\rho\\
&=(\rho=ct\sqrt{y})\\
&=\frac{\Gamma(\frac{n+1}{2}(d-1)+\frac12)}{\Gamma(\frac d2)\Gamma(\frac{n}{2}(d-1))}\int_0^{1}y^{\frac {d}{2}-1}(1-y)^{\frac{n}{2}(d-1)-1}dy=1.
\end{align*}
For the density law $p_{\underline{\bf Y}_d}(\underline{\bf y}_d,t;n)$ similar calculations hold.
\end{remark}
\begin{remark}
From \eqref{condlaw} (for $d=2$) and \eqref{condlaw2} (for $d=4$), we can extract the distribution of $\underline{\bf X}_d(t),t>0,$ and $\underline{\bf Y}_d(t),t>0,$ respectively
\begin{equation}\label{planarlaw}
p_{\underline{\bf X}_2}(\underline{\bf x}_2,t;n)=\frac{n}{2\pi(ct)^n}(c^2t^2-||\underline{\bf x}_2||^2)^{\frac n2-1}
\end{equation}
\begin{equation}\label{fourlaw}
p_{\underline{\bf Y}_4}(\underline{\bf y}_4,t;n)=\frac{n(n+1)}{\pi^2(ct)^{2n+2}}(c^2t^2-||\underline{\bf y}_4||^2)^{n-1}
\end{equation}
which have been obtained in Orsingher and De Gregorio (2007) under the assumption of displacements separated by intervals with uniform joint distribution
$$f(\tau_1,...,\tau_n)=\frac{n!}{t^n}$$
where $0<\tau_j<t-\sum_{k=0}^{j-1}\tau_k,\,1\leq j\leq n,\,\tau_{n+1}=t-\sum_{j=1}^n\tau_j$. In this case an homogenous Poisson process governs the change of orientation of the steps $c\tau_j,\,j=1,...,n$.
\end{remark}
\begin{remark}
From \eqref{condlaw} and $\eqref{condlaw2}$, we are able to derive the exact distribution of a random flight moving in $\mathbb{R}^3$, i.e.
\begin{equation}\label{eq:three1}
p_{\underline{\bf X}_{3}}(\underline{\bf x}_3,t;n)=\frac{\Gamma(n+\frac 32)}{\Gamma(n)}\frac{(c^2t^2-||\underline{\bf x}_3||^2)^{n-1}}{\pi^{\frac 32}(ct)^{2n+1}},\quad n\geq 1,||\underline{\bf x}_3||<ct
\end{equation}
\begin{equation}\label{eq:three2}
p_{\underline{\bf  Y}_{3}}(\underline{\bf y}_3,t;n)=\frac{\Gamma(\frac{n+1}{2}+1)}{\Gamma(\frac n2)}\frac{(c^2t^2-||\underline{\bf y}_3||^2)^{\frac n2-1}}{\pi^{\frac 32}(ct)^{n+1}},\quad n\geq 1,||\underline{\bf y}_3||<ct.
\end{equation}
The following relationship between the distributions  \eqref{eq:three1} and \eqref{eq:three2} emerges
\begin{equation}
p_{\underline{\bf X}_{3}}(\underline{\bf x}_3,t;n)=p_{\underline{\bf Y}_{3}}(\underline{\bf x}_3,t;2n),
\end{equation}
which tells us that a random flight developing in $\mathbb{R}^3$ according to the Dirichlet law $f_1(\tau_1,...,\tau_n)$ has the same density of a three-dimensional random flight where the steps have joint distribution $f_2(\tau_1,...,\tau_{2n})$.

For $n=1$ and $n=2$, the distributions \eqref{eq:three1} and \eqref{eq:three2} provide us the uniform distribution inside the sphere $\mathcal{H}_{ct}^3$(as emerges from Table \ref{unifcond}), that is
$$p_{\underline{\bf X}_{3}}(\underline{\bf x}_3,t;1)=p_{\underline{\bf  Y}_{3}}(\underline{\bf y}_3,t;2)=\frac{\Gamma(\frac 52)}{\pi^{\frac 32}(ct)^3}.$$
Furthermore, we observe that
$$p_{\underline{\bf  Y}_{3}}(\underline{\bf y}_3,t;1)=\frac{1}{\pi^{2}(ct)^{2}}\frac{1}{\sqrt{c^2t^2-||\underline{\bf y}_3||^2}}$$ 
which corresponds to (4.1a) in Orsingher and De Gregorio (2007) for $n=0$. The random traveller for one change of direction is more likely to be near the sphere surface $\partial \mathcal{H}_{ct}^3$ while for two changes of direction his position is uniformly distributed inside the sphere. This has been commented by Franceschetti (2007) for a planar random motion.
\end{remark}

\begin{remark}
It is interesting to note that originally the problem of the random flights has been tackled by considering the length of the steps $c\tau_j$ constant and equal to $\Delta$. Then, in this case the following distribution emerges (compare with Watson, 1922, pag.421, with suitable adjustments of the parameters)
\begin{align*}
p_{\underline{\bf  X}_d}(\underline{\bf x}_d,t;n) =\frac{2^{\left(\frac{d}{2}-1\right)n-1}\pi^{-\frac{d}{2}}\left[\Gamma\left(\frac{d}{2}\right)\right]^{n+1}}{\left(\Delta||\underline{\bf x}_d||^2\right)^{\frac{d}{2}-1}}
\int_0^\infty
\rho^{-\left(\frac{d}{2}-1\right)n+1}J_{\frac{d}{2}-1}\left(\rho||\underline{\bf x}_d||\right)\left(J_{\frac{d}{2}-1}\left(\rho
\Delta\right)\right)^{n+1}d\rho.
\end{align*}
 The above integral can not be worked out. Therefore, as suggested by Theorem \ref{th2}, it is crucial to randomize the length of the steps in order to provide a general solution of the problem of the random flights for the real space having dimension $d$. 
 \end{remark}

The distance from the origin of the position reached by the $d$-dimensional random flights $\underline{\bf X}_d(t),t>0,$ after $n+1$ steps, that is $R_d(t)=||\underline{\bf X}_d(t)||,t>0$,  has the following distribution function
\begin{eqnarray*}
P\{R_d(t)<r\}=\frac{2}{(ct)^{(n+1)(d-1)-1}}\frac{\Gamma(\frac{n+1}{2}(d-1)+\frac12)}{\Gamma(\frac d2)\Gamma(\frac{n}{2}(d-1))}\int_0^r\rho^{d-1}(c^2t^2- \rho^2)^{\frac{n}{2}(d-1)-1}d\rho
\end{eqnarray*}
and therefore the density law of $R_d(t)$ becomes
\begin{equation}\label{radiallaw}
p_{R_d}(r,t;n)=2\frac{\Gamma(\frac{n+1}{2}(d-1)+\frac12)}{\Gamma(\frac d2)\Gamma(\frac{n}{2}(d-1))}\frac{r^{d-1}(c^2t^2- r^2)^{\frac{n}{2}(d-1)-1}}{(ct)^{(n+1)(d-1)-1}}
\end{equation}
with $0<r<ct$.
Analogously, for $L_d(t)=||\underline{\bf Y}_d(t)||,t>0$, we obtain that
\begin{equation}\label{radiallaw2}
p_{L_d}(l,t;n)=2\frac{\Gamma((n+1)(\frac d2-1)+1)}{\Gamma(\frac d2)\Gamma(n(\frac d2-1))}\frac{l^{d-1}(c^2t^2- l^2)^{n(\frac d2-1)-1}}{(ct)^{2(n+1)(\frac d2-1)}}
\end{equation}
with $0<l<ct$. The behavior of densities \eqref{radiallaw} and \eqref{radiallaw2} is outlined in Figure \ref{profile} for $d=3$ and different values of $n$.

  \begin{figure}
  \begin{center}
\includegraphics[angle=0,width=1\textwidth]{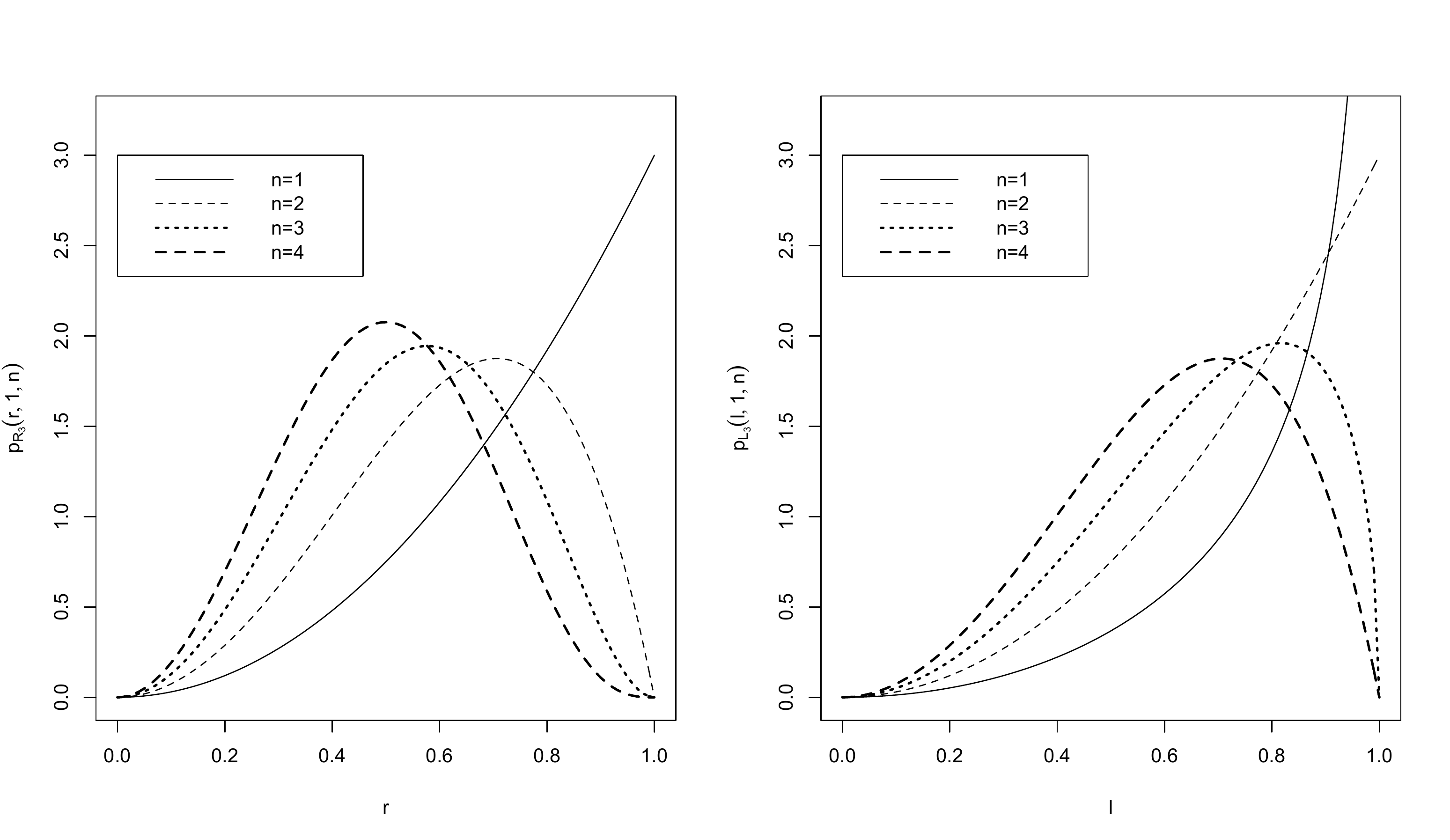}
\caption{The behavior of the densities of $R_3(t)$ and $L_3(t)$, with $c=1$ and $t=1,$ for $n=1,2,3,4$.}\label{profile}
\end{center}
\end{figure}

We present now the expression of moments of the radial processes $R_d(t), L_d(t),t>0$.
\begin{theorem}
For  $p\geq 1$ and $n\geq 1$, we have the following general results
\begin{equation}\label{pmeanR}
E\left\{R_d(t)\right\}^p=\frac{\Gamma(\frac{p+d}{2})\Gamma(\frac{n+1}{2}(d-1)+\frac12)}{\Gamma(\frac d2)\Gamma(\frac{p+d}{2}+\frac n2(d-1))}(ct)^p,\quad d\geq 2,
\end{equation}
and
\begin{equation}\label{pmeanL}
E\left\{L_{d}(t)\right\}^p=\frac{\Gamma(\frac{p+d}{2})\Gamma((n+1)(\frac d2-1)+1)}{\Gamma(\frac d2)\Gamma(\frac{p+d}{2}+n(\frac d2-1))}(ct)^p,\quad d\geq 3.
\end{equation}
\end{theorem}
\begin{proof}
In view of \eqref{radiallaw}, we obtain that

\begin{eqnarray*}
E\left\{R_d(t)\right\}^p&=&\frac{2\Gamma(\frac{n+1}{2}(d-1)+\frac12)}{\Gamma(\frac d2)(ct)^{(n+1)(d-1)-1}\Gamma(\frac n2(d-1))}\int_0^{ct}r^{p+d-1}(c^2t^2-r^2)^{\frac n2(d-1)-1}dr\notag\\
&=&\frac{(ct)^p\Gamma(\frac{n+1}{2}(d-1)+\frac12)}{\Gamma(\frac d2)\Gamma(\frac n2(d-1))}\int_0^1y^{\frac{pp+d}{2}-1}(1-y)^{\frac n2(d-1)-1}dy\notag\\
&=&\frac{\Gamma(\frac{p+d}{2})\Gamma(\frac{n+1}{2}(d-1)+\frac12)}{\Gamma(\frac d2)\Gamma(\frac{p+d}{2}+\frac n2(d-1))}(ct)^p.
\end{eqnarray*}
By similar steps we arrive at \eqref{pmeanL}.
\end{proof}

\begin{remark}
From \eqref{pmeanR} and \eqref{pmeanL}, we can extract the following results
\begin{align}
E\left\{R_d(t)\right\}=\frac{\Gamma(\frac{d-1}{2})}{\Gamma(\frac d2)}\frac{d-1}{n(d-1)+d}ct,
\quad E\left\{R_d(t)\right\}^2
=\frac{d}{n(d-1)+d}(ct)^2,
\end{align}
for $d\geq 2,n\geq 1$, and
\begin{align}
E\left\{L_{d}(t)\right\}
=\frac{\Gamma(\frac{d-1}{2})}{\Gamma(\frac d2)}\frac{d-1}{(n+1)(d-2)+1}ct,\quad E\left\{L_{d}(t)\right\}^2=\frac{d}{n(d-2)+d}(ct)^2,
\end{align}
for $d\geq 3,n\geq 1$. We observe that
\begin{align*}
\frac{E\{R_d(t)\}}{E\{L_d(t)\}}=1-\frac{1}{d-1+\frac{1}{n+1}}, \quad \frac{E\{R_d(t)\}^2}{E\{L_d(t)\}^2}=1-\frac{1}{d-1+\frac{d}{n}}
\end{align*}
and this shows that
\begin{align*}
E\{R_d(t)\}\leq E\{L_d(t)\}, \quad E\{R_d(t)\}^2\leq E\{L_d(t)\}^2
\end{align*}
for all $d\geq 3$.
\end{remark}

\begin{theorem} The projection of the processes $\underline{\bf X}_d(t),t>0,$ and $\underline{\bf Y}_d(t),t>0,$ onto a lower space of dimension $m$, leads to the following marginal distributions
 \begin{equation}\label{eq:marg}
 f_{\underline{\bf X}_m}^d(\underline{\bf x}_m,t;n)=\frac{\Gamma(\frac{(n+1)}{2}(d-1)+\frac12)}{\Gamma(\frac{(n+1)}{2}(d-1)+\frac{1-m}{2})}\frac{(c^2t^2-||\underline{\bf x}_m||^2)^{\frac{(n+1)}{2}(d-1)-\frac{m+1}{2}}}{\pi^{\frac m2}(ct)^{(n+1)(d-1)-1}},
 \end{equation}
 \begin{equation}\label{eq:marg2}
f_{\underline{\bf Y}_m}^d(\underline{\bf y}_m,t;n)= \frac{\Gamma((n+1)(\frac d2-1)+1)}{\Gamma((n+1)(\frac d2-1)+1-\frac m2)}\frac{(c^2t^2- ||\underline{\bf y}_{m}||^2)^{(n+1)(\frac d2-1)-\frac m2}}{\pi^{\frac m2}(ct)^{2(n+1)(\frac d2-1)}},
   \end{equation}
 with $||\underline{\bf x}_m||<ct$, $||\underline{\bf y}_m||<ct$ and $1\leq m< d$. For $m=d$ the densities \eqref{eq:marg} and \eqref{eq:marg2} coincide with \eqref{condlaw} and \eqref{condlaw2}.
 \end{theorem}
 \begin{proof}
We start by observing that the projection of the random process $\underline{\bf X}_d(t),t>0,$ onto the space $\mathbb{R}^m$, represents a random flight with $m$ components having density law given by
 \begin{align*}
 f_{\underline{\bf X}_m}^d(\underline{\bf x}_m,t;n)=\int_{-\sqrt{c^2t^2-||\underline{\bf x}_m||^2}}^{\sqrt{c^2t^2-||\underline{\bf x}_m||^2}}dx_{m+1}\cdots\int_{-\sqrt{c^2t^2-||\underline{\bf x}_{d-2}||^2}}^{\sqrt{c^2t^2-||\underline{\bf x}_{d-2}||^2}}dx_{d-1}\int_{-\sqrt{c^2t^2-||\underline{\bf x}_{d-1}||^2}}^{\sqrt{c^2t^2-||\underline{\bf x}_{d-1}||^2}} p_{\underline{\bf  X}_d}(\underline{\bf x}_d,t;n)dx_{d}.
\end{align*}
Then
\begin{align*}
&\int_{-\sqrt{c^2t^2-||\underline{\bf x}_{d-1}||^2}}^{\sqrt{c^2t^2-||\underline{\bf x}_{d-1}||^2}}p_{\underline{\bf X}_d}(\underline{\bf x}_d,t;n) dx_{d}\\
&=\frac{2}{\pi^{d/2}(ct)^{(n+1)(d-1)-1}}\frac{\Gamma(\frac{n+1}{2}(d-1)+\frac12)}{\Gamma(\frac{n}{2}(d-1))}\int_{0}^{\sqrt{c^2t^2-||\underline{\bf x}_{d-1}||^2}}(c^2t^2- ||\underline{\bf x}_d||^2)^{\frac{n}{2}(d-1)-1}dx_d\\
&=(x_d=\sqrt{w}\sqrt{c^2t^2-||\underline{\bf x}_{d-1}||^2})\\
&=\frac{\Gamma(\frac{n+1}{2}(d-1)+\frac12)}{\pi^{d/2}(ct)^{(n+1)(d-1)-1}}\frac{(c^2t^2- ||\underline{\bf x}_{d-1}||^2)^{\frac{n}{2}(d-1)-\frac12}}{\Gamma(\frac{n}{2}(d-1))}\int_{0}^1w^{-\frac12}(1-w)^{\frac{n}{2}(d-1)-1}dw\\
&=\frac{1}{\pi^{d/2}(ct)^{(n+1)(d-1)-1}}\frac{\Gamma(\frac{n+1}{2}(d-1)+\frac12)\sqrt{\pi}}{\Gamma(\frac{n}{2}(d-1)+\frac12)}(c^2t^2- ||\underline{\bf x}_{d-1}||^2)^{\frac{n}{2}(d-1)-\frac12}.
\end{align*}
For the integral with respect to $x_d$, we have that
 \begin{align*}
& \frac{2}{\pi^{d/2}(ct)^{(n+1)(d-1)-1}}\frac{\Gamma(\frac{n+1}{2}(d-1)+\frac12)\sqrt{\pi}}{\Gamma(\frac{n}{2}(d-1)+\frac12)}\int_{0}^{\sqrt{c^2t^2-||\underline{\bf x}_{d-2}||^2}}(c^2t^2- ||\underline{\bf x}_{d-1}||^2)^{\frac{n}{2}(d-1)-\frac12}dx_{d-1}\\
&=(x_{d-1}=\sqrt{w}\sqrt{c^2t^2-||\underline{\bf x}_{d-2}||^2})\\
&= \frac{1}{\pi^{d/2}(ct)^{(n+1)(d-1)-1}}\frac{\Gamma(\frac{n+1}{2}(d-1)+\frac12)\sqrt{\pi}}{\Gamma(\frac{n}{2}(d-1)+\frac12)}(c^2t^2- ||\underline{\bf x}_{d-2}||^2)^{\frac{n}{2}(d-1)}\int_0^1w^{-\frac12}(1-w)^{\frac{n}{2}(d-1)-\frac12}dw\\
&= \frac{1}{\pi^{d/2}(ct)^{(n+1)(d-1)-1}}\frac{\Gamma(\frac{n+1}{2}(d-1)+\frac12)\pi}{\Gamma(\frac{n}{2}(d-1)+1)}(c^2t^2- ||\underline{\bf x}_{d-2}||^2)^{\frac{n}{2}(d-1)}.
 \end{align*}
 Therefore calculating the successive integrals in this way, we obtain that
 $$ f_{\underline{\bf X}_m}^d(\underline{\bf x}_m,t;n)=\frac{\Gamma(\frac{n+1}{2}(d-1)+\frac12)\pi^{\frac{d-m}{2}}}{\Gamma(\frac{n}{2}(d-1)+\frac{d-m}{2})} \frac{(c^2t^2- ||\underline{\bf x}_{m}||^2)^{\frac{n}{2}(d-1)-1+\frac{d-m}{2}}}{\pi^{d/2}(ct)^{(n+1)(d-1)-1}},$$
with $ ||\underline{\bf x}_{m}||<ct$, and by simple manipulations the expression \eqref{eq:marg} emerges.
     
 By using the same approach we derive the result \eqref{eq:marg2} concerning the projection of the random flights $\underline{\bf Y}_d(t),t>0$.
\end{proof}

 \begin{remark}
The functions of the form
 $$q(\underline{\bf x}_d,t)=\left(c^2t^2-||\underline{\bf x}_d||^2\right)^m,\quad ||\underline{\bf x}_d||<ct,$$
 appearing in all densities \eqref{condlaw}, \eqref{condlaw2}, \eqref{eq:marg} and \eqref{eq:marg2}, are solutions to the telegraph-type equations
\begin{equation}\label{telegraphtype}
\frac{\partial^2 q}{\partial t^2}=c^2\sum_{j=1}^d\frac{\partial^2 q}{\partial x_j^2}+\frac{2m-1+d}{t}\frac{\partial q}{\partial t}.
\end{equation}
In particular, for $m=-\frac{d-1}{2}$ equation \eqref{telegraphtype} simplifies and becomes the $d$-dimensional wave equation
$$\frac{\partial^2 q}{\partial t^2}=c^2\sum_{j=1}^d\frac{\partial^2 q}{\partial x_j^2}.$$
 \end{remark}
  
  \begin{remark}
  
 From \eqref{eq:marg}, for $d=2$ and $m=1$, we obtain the marginal density of the conditional distribution of a planar random flight with uniformly distributed switching times, that is
$$f_{X_1}^2(x_1,t;n)=\frac{\Gamma(\frac n2)\Gamma(\frac n2+1)}{2\pi\Gamma(n)}\left(\frac{2}{ct}\right)^n(c^2t^2-x_1^2)^{\frac{n-1}{2}},\quad n\geq 1,\,|x_1|<ct.$$
Furthermore, by setting $d=4$ and $m=1,2,3$ in the densities  \eqref{eq:marg} and \eqref{eq:marg2}, we derive the probability distributions of the projections of a four-dimensional random flight, respectively $\underline{\bf X}_4(t)$ and $\underline{\bf Y}_4(t)$, onto the lower spaces (see Table \ref{projection}). Analogously, Table \ref{projection2} summarizes the marginal density laws of the three-dimensional random flights onto the spaces $\mathbb{R}^2$ and $\mathbb{R}$.

From Table \ref{projection} emerges that for $n=2r+1,r=0,1,...$ the distribution of a planar random flight (with intertimes having uniform law) $p_{\underline{\bf X}_{2}}(\underline{\bf x}_{2},t;n)$ coincides with $f_{\underline{\bf X}_{2}}^3(\underline{\bf x}_{2},t;n)$. Furthermore, the distribution $p_{\underline{\bf Y}_{4}}(\underline{\bf y}_{4},t;n)$ coincides with the result obtained in Orsingher and De Gregorio (2007), formula (3.2). Therefore, the marginal densities of $\underline{\bf Y}_{4}(t),t>0,$ of the Table below coincide with the probability distributions (4.1a), (4.1b) and (4.1c) in  Orsingher and De Gregorio (2007).
 
 \begin{table}[h]
\begin{center}
\begin{tabular}{c|c|c}
$d=4$&$\underline{\bf X}_4(t)$& $\underline{\bf Y}_4(t)$\\		
\hline&&\\
$m=3$&$\frac{(\Gamma(\frac32n+2))^22^{3n+3}}{\pi^2(ct)^{3n+2}\Gamma(3n+4)}
(c^2t^2- ||\underline{\bf x}_{3}||^2)^{\frac32n-\frac{1}{2}}$& $\frac{\Gamma(n+2)\Gamma(n)2^{2n-1}}{\pi^2(ct)^{2n+2}\Gamma(2n)}
(c^2t^2- ||\underline{\bf y}_{3}||^2)^{n-\frac{1}{2}}$\\&&\\
$m=2$&$\frac{\frac32n+1}{\pi(ct)^{3n+2}}
(c^2t^2- ||\underline{\bf x}_{2}||^2)^{\frac32n}$& $\frac{(n+1)}{\pi(ct)^{2n+2}}
(c^2t^2-||\underline{\bf y}_{2}||^2)^n$\\&&\\
$m=1$&$\frac{\Gamma(\frac32n+2)\Gamma(\frac32n+1)2^{3n+1}}{\pi(ct)^{3n+2}\Gamma(3n+2)}
(c^2t^2-  x_{1}^2)^{\frac32n+\frac{1}{2}}$&$\frac{\Gamma(n+2)\Gamma(n+1)2^{2n+1}}{\pi(ct)^{2n+2}\Gamma(2n+2)}
(c^2t^2-y_1^2)^{n+\frac{1}{2}}$
\end{tabular}
\end{center}
\caption{The density laws of the processes representing the projections onto the lower spaces of $\underline{\bf X}_4(t),t>0,$ and $\underline{\bf Y}_4(t),t>0$.}\label{projection}
\end{table}%

 \begin{table}[h]
\begin{center}
\begin{tabular}{c|c|c}
$d=3$&$\underline{\bf X}_3(t)$& $\underline{\bf Y}_3(t)$\\		
\hline&&\\
$m=2$&$\frac{n+\frac12}{\pi(ct)^{2n+1}}
(c^2t^2- ||\underline{\bf x}_{2}||^2)^{n-\frac12}$& $\frac{(n+1)}{2\pi(ct)^{n+1}}
(c^2t^2-||\underline{\bf y}_{2}||^2)^{\frac{n-1}{2}}$\\&&\\
$m=1$&$\frac{\Gamma(n+\frac32)}{\sqrt{\pi}(ct)^{2n+1}\Gamma(n+1)}
(c^2t^2-  x_{1}^2)^{n}$&$\frac{\Gamma(\frac n2+\frac32)}{\sqrt{\pi}(ct)^{n+1}\Gamma(\frac n2+1)}
(c^2t^2-y_1^2)^{\frac{n}{2}}$
\end{tabular}
\end{center}
\caption{The density laws of the processes representing the projections onto the lower spaces of $\underline{\bf X}_3(t),t>0,$ and $\underline{\bf Y}_3(t),t>0$.}\label{projection2}
\end{table}%
\end{remark}

We are able to provide an alternative stochastic representation of the motion on the real line representing the projection of the $d$-dimensional random flight $\underline{\bf X}_{d}(t),t>0,$ as follows. We consider a random walker moving, with velocity $c>0$, forward for a time $tT_n^d$, where $T_n^d$ is a $Beta(\frac{n+1}{2}(d-1),\frac{n+1}{2}(d-1))$ random variable, and backward for the remaining time $t(1-T_n^d)$. In other words, the position $\hat X(t),t>0,$ reached by the random traveller at time $t$ is equal to
\begin{equation}\label{eq:represX}
\hat X(t)=ct\left[T_n^d-(1-T_n^d)\right].
\end{equation}

Therefore, the density law of $\hat X(t),t>0,$ becomes
\begin{eqnarray}\label{eq:marg1}
p_{\hat X}(x,t;n)&=&\frac{d}{dx}P\left\{\hat X(t)<x\right\}=\frac{d}{dx}P\left\{T_n^d<\frac{2ct+x}{2ct}\right\}\notag\\
&=&\frac{\Gamma((n+1)(d-1))}{(\Gamma(\frac{n+1}{2}(d-1)))^2}\frac{d}{dx}\int_0^{\frac{2ct+x}{2ct}}w^{\frac{n+1}{2}(d-1)-1}(1-w)^{\frac{n+1}{2}(d-1)-1}dw\notag\\
&=&\frac{1}{(2ct)^{(n+1)(d-1)-1}}\frac{\Gamma((n+1)(d-1))}{(\Gamma(\frac{n+1}{2}(d-1)))^2}(c^2t^2-x^2)^{\frac{n+1}{2}(d-1)-1}
\end{eqnarray}
and coincides with the distribution \eqref{eq:marg} with $m=1$. We observe that from \eqref{eq:marg1} for $d=3$ and $m=1$, we have that 
$$f_{X_1}^3(x_1,t;n)=\frac{P\{T(t)\in dx_1|N(t)=2n+1\}}{dx_1}=\frac{(2n+1)!}{(n!)^2}\frac{1}{(2ct)^{2n+1}}(c^2t^2-x_1^2)^n$$
where the duplication formula for Gamma functions has been applied and $T(t),t>0,$ represents the standard telegraph process. The above distribution has been obtained by De Gregorio {\it et al.} (2005) (see formula (2.17)) by applying the order statistics.

We can represent the one-dimensional motion underlying distribution \eqref{eq:marg1} as a sequence of alternating forward and backward displacements at speed $c$, where the change of direction occurs at Poisson times $t_j$ with $\tau_j=t_j-t_{j-1},\,j=1,...,n$. Since the intertimes $t_j,\,j=1,...,n$ possess uniform joint distribution, in force of exchangeability we can rearrange the displacements and put together forward steps and backward ones. In our case, we have for $d=2r+1,\,r=0,1,...,$ $(n+1)r$ forward displacements and an equal number of backward ones. The instant $T_{(n+1)r}^{2r+1}$ indicates the time where the last forward displacement occurs in the rearranged sequence and corresponds to the $(n+1)r$-th order statistics from a uniform distribution in $(0,t)$. This leads to the representation \eqref{eq:represX}.

The $d$-dimensional random flight producing in $\mathbb{R}^2$ the planar motion described, for example, in Stadje (1987), with distribution \eqref{planarlaw}, must satisfy the following relationship 
$$r=(n+1)(d-1)-1$$
among $d$, the number $r$ of changes of direction $n$ in $\mathbb{R}^d$ and the number of changes of direction of the probabilistically equivalent motion in $\mathbb{R}^2$ (in \eqref{eq:marg} write $(n+1)(d-1)=r+1$ and $m=2$ so that \eqref{planarlaw} emerges).

Analogously, for the projection onto $\mathbb{R}^1$ of $\underline{\bf Y}_{d}(t),t>0$, a similar representation holds true. Indeed, in this case we can write
\begin{equation}
\hat Y(t)=ct\left[S_n^d-(1-S_n^d)\right]
\end{equation}
where $S_n^d$ is a $Beta((n+1)(\frac d2-1)+\frac12,(n+1)(\frac d2-1)+\frac12)$ random variable. Then, we obtain that
\begin{align}
p_{\hat Y}(y,t;n)&=\frac{d}{dy}P\left\{\hat Y(t)<y\right\}=\frac{d}{dy}P\left\{S_n^d<\frac{2ct+y}{2ct}\right\}\\
&=\frac{\Gamma(2(n+1)(\frac d2-1)+1)}{(\Gamma((n+1)(\frac d2-1)+\frac12)))^2}\frac{(c^2t^2-y^2)^{(n+1)(\frac d2-1)-\frac 12}}{(2ct)^{2(n+1)(\frac d2-1)}}\notag
\end{align}
which coincides with the distribution \eqref{eq:marg2} with $m=1$.

\begin{remark}
The density  $f_{\underline{\bf X}_{m}}^d(\underline{\bf x}_{m},t;n)$ becomes uniform on the hypershere $\mathcal{H}_{ct}^m$ when
$
d=\frac{n+m+2}{n+1}, \, 1\leq m\leq d.
$
Analogously, the distribution \eqref{eq:marg2} is uniform on the hypershere $\mathcal{H}_{ct}^m$ if 
$
d=\frac{2n+m+2}{n+1}, \, 1\leq m\leq d
$
Admissible combinations of $m$ and $n$ are those for which the dimension $d$ is an integer number (see Table \ref{unifcond}). Some simulations presented in Le Caer (2010) give graphic evidence of the uniform law for $n=2$ and $d=2$.
\end{remark}

 \begin{table}[h]
\begin{center}
\begin{tabular}{c|c}
$\underline{\bf X}_d(t)$& $\underline{\bf Y}_d(t)$\\		
\hline\\
$n=2,\,d=m=2$&$n=1,\,m=2,\,d=3$\\
$n=m=1,\, d=2$&$n=2,\,d=m=3$\\
$n=1,\,d=m=3$&$n=1,\,d=m=4$
\end{tabular}
\end{center}
\caption{The values of $n,\,d$ and $m$ leading to the uniform distribution.}\label{unifcond}
\end{table}%

\begin{remark}
We expect that the density tends to infinity near the surface of the hypersphere $\mathcal{H}_{ct}^m$ for a small number $n$ of deviations and this is confirmed by the following inequalities
$
n<\frac{m-d+2}{d-1},
$
for $\underline{\bf X}_{d}(t),t>0$, and
$
n<\frac{m-d+2}{d-2},$
 for $\underline{\bf Y}_{d}(t),t>0$.
Since $m\leq d$ and $n$ is integer, we obtain that the first inequality holds for $n=1$ and $d=m=2$, and the second inequality is valid for $n=1$ and $d=m=3$. In all the remaining cases the distributions take a bell-shaped structure because the larger is the number of changes of orientation, the shorter become the displacements and the closer to the origin is the moving particle.
\end{remark}

\section{Unconditional probability distributions}

In order to obtain unconditional densities for $\underline{\bf X}_d(t),t>0,$ we randomize the number of deviations $\mathcal{N}_d(t)$, at time $t>0$, by assuming that it possesses the distribution of a fractional Poisson process (consult on this point Beghin and Orsingher, 2009).  In this context by fractional Poisson process we mean a process with distribution
\begin{equation}\label{lawgenpoi}
P\left\{\mathcal{N}_d(t)=n\right\}=\frac{1}{E_{\frac{d-1}{2},\frac d2}(\lambda t)}\frac{(\lambda t)^n}{\Gamma((\frac{d-1}{2})n+\frac d2)},\quad d\geq 2,\, n=0,1,...
\end{equation}
where $E_{\alpha,\beta}(x)=\sum_{k=0}^\infty \frac{x^k}{\Gamma(\alpha k+\beta)},\, x\in\mathbb{R},\alpha,\beta>0,$ is the generalized Mittag-Leffler function. The generating function of the probabilities is
$$G_{\mathcal{N}_d}(u,t)=\frac{E_{\frac{d-1}{2},\frac d2}(\lambda tu)}{E_{\frac{d-1}{2},\frac d2}(\lambda t)},\quad |u|\leq 1.$$
Since
\begin{equation}\label{eq:rel}
\frac{d}{dx}E_{\nu,\beta}(ax)=\frac a\nu\left[E_{\nu,\nu+\beta-1}(ax)+(1-\beta)E_{\nu,\nu+\beta}(ax)\right]
\end{equation}
we obtain that
\begin{eqnarray}
E\left\{\mathcal{N}_d(t)\right\}&=&\frac{d}{du}G_\mathcal{N}(u,t)|_{u=1}\notag\\
&=&\frac{2\lambda t}{(d-1)E_{\frac{d-1}{2},\frac d2}(\lambda t)}\left[E_{\frac{d-1}{2},d-\frac 32}(\lambda t)+\left(1-\frac d2\right)E_{\frac{d-1}{2},d-\frac 12}(\lambda t)\right]
\end{eqnarray}
The above result can also be obtained directly by using
$$E\left\{\mathcal{N}_d(t)\right\}=\frac{2}{d-1}\frac{1}{E_{\frac{d-1}{2},\frac d2}(\lambda t)}\left[\sum_{n=1}^\infty \frac{(\lambda t)^n}{\Gamma((\frac{d-1}{2})n+\frac d2)}\left(n\left(\frac{d-1}{2}\right)+\frac d2-1-\frac d2+1\right)\right]$$
and by performing some straightforward calculations. If $d=2$ the mean value of $\mathcal{N}_d(t)$ becomes
\begin{eqnarray*}
E\left\{\mathcal{N}_2(t)\right\}
=\frac{2\lambda t}{E_{\frac{1}{2},1}(\lambda t)}E_{\frac{1}{2},\frac 12}(\lambda t)=\text{(from \eqref{eq:rel})}
=\frac{ t}{E_{\frac{1}{2},1}(\lambda t)}\frac{d}{dt}E_{\frac{1}{2},1}(\lambda t)=t\frac{d}{dt}\log E_{\frac{1}{2},1}(\lambda t)
\end{eqnarray*}
while for $d=3$, we get that
\begin{equation}\label{mean3}
E\left\{\mathcal{N}_3(t)\right\}=\lambda t-\frac{\lambda t}{2} \frac{ E_{1,\frac{5}{2}}(\lambda t)}{ E_{1,\frac{3}{2}}(\lambda t)} .
\end{equation}
Result \eqref{mean3} shows that the fractional Poisson process with distribution \eqref{lawgenpoi} has a mean number of events growing more slowly than the classical Poisson one.

Analogously, for $\underline{\bf Y}_d(t),t>0,$ we represent the random number of deviations by means of the process $\mathcal{M}^d(t),t>0,$ having probability distribution
\begin{equation}\label{lawgenpoi2}
P\left\{\mathcal{M}_d(t)=n\right\}=\frac{1}{E_{\frac{d}{2}-1,\frac d2}(\lambda t)}\frac{(\lambda t)^n}{\Gamma((\frac{d}{2}-1)n+\frac d2)},\quad d\geq 3,\,n=0,1,...
\end{equation}
The generating function of the probabilities in this case reads
$$G_{\mathcal{M}_d}(u,t)=\frac{E_{\frac{d}{2}-1,\frac d2}(\lambda tu)}{E_{\frac{d}{2}-1,\frac d2}(\lambda t)},\quad |u|\leq 1.$$
From relationship \eqref{eq:rel} emerges that
\begin{eqnarray}
E\left\{\mathcal{M}_d(t)\right\}&=&\frac{d}{du}G_{\mathcal{M}_d}(u,t)|_{u=1}\notag\\
&=&\frac{2\lambda t}{(d-2)E_{\frac{d}{2}-1,\frac d2}(\lambda t)}\left[E_{\frac{d}{2}-1,d-2}(\lambda t)+\left(1-\frac d2\right)E_{\frac{d}{2}-1,d-1}(\lambda t)\right]
\end{eqnarray}
which for $d=4$ becomes
\begin{equation}\label{mean4}
E\left\{\mathcal{M}_4(t)\right\}=\lambda t\left(1- \frac{ E_{1,3}(\lambda t)}{ E_{1,2}(\lambda t)}\right) .
\end{equation}

For $\mathcal{N}_d(t)=0$ the particle reaches the surface of the hypersphere with probability
\begin{equation}
P\{\underline{\bf X}_{d}(t)\in\partial \mathcal{H}_{ct}^d\}=\frac{1}{E_{\frac{d-1}{2},\frac d2}(\lambda t)}\frac{1}{\Gamma(\frac d2)}
\end{equation}
while if $\mathcal{M}_d(t)=0$, one has that
\begin{equation}
P\{\underline{\bf Y}_{d}(t)\in\partial \mathcal{H}_{ct}^d\}=\frac{1}{E_{\frac{d}{2}-1,\frac d2}(\lambda t)}\frac{1}{\Gamma(\frac d2)}.
\end{equation}

We assume that the processes $\mathcal{N}_d(t)$ and $\mathcal{M}_d(t),t>0,$ are independent from the Dirichlet distributed displacements and the angle orientations. Now, we provide the main results of this Section.

\begin{theorem}\label{teounc}
If the number of deviations is represented by a fractional Poisson process with distribution \eqref{lawgenpoi} then the absolutely continuous component of the probability distributions of  $\underline{\bf X}_{d}(t),t>0,$ is equal to
\begin{equation}\label{unclaw1}
\frac{P\{\underline{\bf X}_{d}(t)\in d\underline{\bf x}_{d}\}}{\prod_{j=1}^ddx_j}=\frac{\lambda t}{\pi^{\frac d2}}\frac{(c^2t^2-||\underline{\bf x}_{d}||^2)^{\frac{d-1}{2}-1}}{(ct)^{2(d-1)-1}}\frac{E_{\frac {d-1}{2},\frac {d-1}{2}}\left(\frac{\lambda t(c^2t^2-||\underline{\bf x}_{d}||^2)^{\frac {d-1}{2}}}{(ct)^{(d-1)}}\right)}{E_{\frac{d-1}{2},\frac d2}(\lambda t)}
\end{equation}
where $d\geq 2$, $||\underline{\bf x}_{d}||<ct$, while if the number of the deviations is given by \eqref{lawgenpoi2} , the distribution of $\underline{\bf Y}_{d}(t),t>0,$ reads
\begin{equation}\label{unclaw2}
\frac{P\{\underline{\bf Y}_{d}(t)\in d\underline{\bf y}_{d}\}}{\prod_{j=1}^ddy_j}=\frac{\lambda t}{\pi^{\frac d2}}\frac{(c^2t^2-||\underline{\bf y}_{d}||^2)^{\frac{d}{2}-2}}{(ct)^{4(\frac d2-1)}}\frac{E_{\frac {d}{2}-1,\frac {d}{2}-1}\left(\frac{\lambda t(c^2t^2-||\underline{\bf y}_{d}||^2)^{\frac {d}{2}-1}}{(ct)^{(d-2)}}\right)}{E_{\frac{d}{2}-1,\frac d2}(\lambda t)}
\end{equation}
where $d\geq 3$, $||\underline{\bf y}_{d}||<ct$.
\end{theorem}
\begin{proof} For the random flight $\underline{\bf X}_{d}(t),t>0$, we have that
\begin{align*}
\frac{P\{\underline{\bf X}_{d}(t)\in d\underline{\bf x}_{d}\}}{\prod_{j=1}^ddx_j}
&=\sum_{n=1}^\infty p_{\underline{\bf X}_d}(\underline{\bf x}_d,t;n) \,P\left\{\mathcal{N}_d(t)=n\right\}\notag\\
&=\frac{1}{\pi^{\frac d2}}\frac{1}{E_{\frac{d-1}{2},\frac d2}(\lambda t)}\sum_{n=1}^\infty\frac{(\lambda t)^n}{\Gamma(\frac n2(d-1))}\frac{(c^2t^2-||\underline{\bf x}_{d}||^2)^{\frac n2(d-1)-1}}{(ct)^{(n+1)(d-1)-1}}\notag\\
&=\frac{1}{\pi^{\frac d2}}\frac{1}{E_{\frac{d-1}{2},\frac d2}(\lambda t)}\sum_{n=0}^\infty\frac{(\lambda t)^{n+1}}{\Gamma(\frac{ n+1}{2}(d-1))}\frac{(c^2t^2-||\underline{\bf x}_{d}||^2)^{\frac {n+1}{2}(d-1)-1}}{(ct)^{(n+2)(d-1)-1}}\notag\\
&=\frac{1}{\pi^{\frac d2}}\frac{\lambda t}{E_{\frac{d-1}{2},\frac d2}(\lambda t)}\frac{(c^2t^2-||\underline{\bf x}_{d}||^2)^{\frac{d-1}{2}-1}}{(ct)^{2(d-1)-1}}\sum_{n=0}^\infty\frac{1}{\Gamma(\frac{ n+1}{2}(d-1))}\left[\frac{\lambda t(c^2t^2-||\underline{\bf x}_{d}||^2)^{\frac {d-1}{2}}}{(ct)^{(d-1)}}\right]^n\notag\\
&=\frac{\lambda t}{\pi^{\frac d2}}\frac{(c^2t^2-||\underline{\bf x}_{d}||^2)^{\frac{d-1}{2}-1}}{(ct)^{2(d-1)-1}}\frac{E_{\frac {d-1}{2},\frac {d-1}{2}}\left(\frac{\lambda t(c^2t^2-||\underline{\bf x}_{d}||^2)^{\frac {d-1}{2}}}{(ct)^{(d-1)}}\right)}{E_{\frac{d-1}{2},\frac d2}(\lambda t)}\notag,
\end{align*}

Similarly, for the random flight $\underline{\bf Y}_{d}(t),t>0,$ we are able to derive the density \eqref{unclaw2} as follows
\begin{align*}
\frac{P\{\underline{\bf Y}_{d}(t)\in d\underline{\bf y}_{d}\}}{\prod_{j=1}^ddy_j}
&=\sum_{n=1}^\infty p_{\underline{\bf Y}_d}(\underline{\bf y}_d,t;n)\,P\left\{\mathcal{M}_d(t)=n\right\}\notag\\
&=\frac{1}{\pi^{\frac d2}}\frac{1}{E_{\frac{d}{2}-1,\frac d2}(\lambda t)}\sum_{n=1}^\infty\frac{(\lambda t)^n}{\Gamma(n(\frac d2-1))}\frac{(c^2t^2-||\underline{\bf y}_{d}||^2)^{n(\frac d2-1)-1}}{(ct)^{2(n+1)(\frac d2-1)}}\notag\\
&=\frac{1}{\pi^{\frac d2}}\frac{1}{E_{\frac{d}{2}-1,\frac d2}(\lambda t)}\sum_{n=0}^\infty\frac{(\lambda t)^{n+1}}{\Gamma((n+1)(\frac d2-1))}\frac{(c^2t^2-||\underline{\bf y}_{d}||^2)^{(n+1)(\frac d2-1)-1}}{(ct)^{2(n+2)(\frac d2-1)}}\notag\\
&=\frac{1}{\pi^{\frac d2}}\frac{\lambda t}{E_{\frac{d}{2}-1,\frac d2}(\lambda t)}\frac{(c^2t^2-||\underline{\bf y}_{d}||^2)^{\frac{d}{2}-2}}{(ct)^{4(\frac d2-1)}}\sum_{n=0}^\infty\frac{1}{\Gamma((n+1)(\frac d2-1))}\left[\frac{\lambda t(c^2t^2-||\underline{\bf y}_{d}||^2)^{\frac {d}{2}-1}}{(ct)^{(d-2)}}\right]^n\notag\\
&=\frac{\lambda t}{\pi^{\frac d2}}\frac{(c^2t^2-||\underline{\bf y}_{d}||^2)^{\frac{d}{2}-2}}{(ct)^{4(\frac d2-1)}}\frac{E_{\frac {d}{2}-1,\frac {d}{2}-1}\left(\frac{\lambda t(c^2t^2-||\underline{\bf y}_{d}||^2)^{\frac {d}{2}-1}}{(ct)^{(d-2)}}\right)}{E_{\frac{d}{2}-1,\frac d2}(\lambda t)}\notag.
\end{align*}

\end{proof}

\begin{remark}
It is not hard to verify that
$$\int_{\mathcal{H}_{ct}^d}P\{\underline{\bf X}_{d}(t)\in d\underline{\bf x}_{d}\}=1-P\{\underline{\bf X}_{d}(t)\in \partial\mathcal{H}_{ct}^d\}=1-\frac{1}{E_{\frac{d-1}{2},\frac d2}(\lambda t)}\frac{1}{\Gamma(\frac d2)}
$$
and
$$\int_{\mathcal{H}_{ct}^d}P\{\underline{\bf Y}_{d}(t)\in d\underline{\bf y}_{d}\}=1-P\{\underline{\bf Y}_{d}(t)\in \partial\mathcal{H}_{ct}^d\}=1-\frac{1}{E_{\frac{d}{2}-1,\frac d2}(\lambda t)}\frac{1}{\Gamma(\frac d2)}.
$$
\end{remark}

\begin{remark}
We examine some particular cases where the densities of Theorem \ref{teounc} take interesting forms. Indeed, by observing that $E_{1,1}(x)=e^x$ and $E_{1,2}(x)=\frac {e^x-1}{x}$, we have the distributions summarized in the following Table.

 \begin{table}[h]
\begin{center}
\begin{tabular}{c|cc}
&$d=2$& $d=3$\\\hline\\	
$\underline{\bf X}_d(t)$&$\frac{\lambda}{\pi cE_{\frac12,1}(\lambda t)}\frac{E_{\frac12,\frac12}\left(\frac{\lambda}{c}\sqrt{c^2t^2-||\underline{\bf x}_{2}||^2}\right)}{\sqrt{c^2t^2-||\underline{\bf x}_{2}||^2}}$&$\frac{\lambda}{\pi^{\frac32}c^3t^2}\frac{e^{\frac{\lambda}{c^2t}(c^2t^2-||\underline{\bf x}_{3}||^2)}}{E_{1,\frac32}(\lambda t)}$\\\\
&$d=3$&$d=4$\\\hline\\
 $\underline{\bf Y}_d(t)$&$\frac{\lambda }{\pi^{\frac 32}c^2t}\frac{1}{E_{\frac{1}{2},\frac 32}(\lambda t)}\frac{E_{\frac {1}{2},\frac {1}{2}}\left(\frac{\lambda}{c} \sqrt{c^2t^2-||\underline{\bf y}_{3}||^2}\right)}{\sqrt{c^2t^2-||\underline{\bf y}_{3}||^2}}$&$\frac{\lambda^2}{\pi^{2}c^4t^2}\frac{e^{\frac{\lambda}{c^2t}(c^2t^2-||\underline{\bf y}_{4}||^2)}}{e^{\lambda t}-1}$\\
\end{tabular}
\end{center}
\end{table}

All the distributions \eqref{unclaw1}, \eqref{unclaw2} and the special cases in the above Table have the isotropic form \eqref{eq:introd2}.

If we suppose that the changes of direction are governed by an homogeneous Poisson process the intervals $\tau_1,...,\tau_n,$ are uniformly distributed on $[0,t]$. In this case the absolutely continuous component of the unconditional distribution of a planar random flight $\underline{\bf X}_2(t),t>0,$ is given by (see Stadje, 1987)
\begin{equation}\label{stadje}
\overline{p}_{\underline{\bf X}_2}(\underline{\bf x}_2,t)=\frac{\lambda e^{-\lambda t}}{2\pi c}\frac{e^{\frac{\lambda}{c}\sqrt{c^2t^2-||\underline{\bf x}_2||^2}}}{\sqrt{c^2t^2-||\underline{\bf x}_2||^2}},\quad ||\underline{\bf x}_2||<ct.
\end{equation}
Therefore, by comparing \eqref{stadje} with the corresponding distribution in the above Table, we observe that in our context, $e^{\lambda t}/2$ and $e^{\frac{\lambda}{c}\sqrt{c^2t^2-||\underline{\bf x}_2||^2}}$ are replaced by $E_{\frac12,1}(\lambda t)$ and $E_{\frac12,\frac12}\left(\frac{\lambda}{c}\sqrt{c^2t^2-||\underline{\bf x}_{2}||^2}\right)$, respectively.

Furthermore, the absolutely continuous part of the distribution of a four-dimensional random flight with Poissonian switching times  (see formula (3.7) in Orsingher and De Gregorio, 2007) is equal to
\begin{equation*}
\overline p_{\underline{\bf Y}_4}(\underline{\bf y}_4,t)=\frac{\lambda}{c^4t^3\pi^2}e^{-\frac{\lambda}{c^2t}||\underline{\bf y}_{4}||^2}\left\{2+\frac{\lambda}{c^2t}(c^2t^2-||\underline{\bf y}_{4}||^2)\right\}
\end{equation*}
which has to be compared with
\begin{equation*}
 p_{\underline{\bf Y}_4}(\underline{\bf y}_4,t)=\frac{\lambda^2}{\pi^{2}c^4t^2}\frac{e^{\frac{\lambda}{c^2t}(c^2t^2-||\underline{\bf y}_{4}||^2)}}{e^{\lambda t}-1}.
\end{equation*}
\end{remark}

We indicate with $\underline{\bf X}_{m}^d(t)=(X_1(t),...,X_m(t)),t>0,$ and with $\underline{\bf Y}_{m}^d(t)=(Y_1(t),...,Y_m(t)),t>0,$ the random processes emerging from the projection onto $\mathbb{R}^m$ of $\underline{\bf X}_d(t)$ and $\underline{\bf Y}_{d}(t)$, respectively. In the next Theorem we give the unconditional distribution of  $\underline{\bf X}_{m}^d(t),t>0,$ and  $\underline{\bf Y}_{m}^d(t),t>0,$ $1\leq m<d$. The singular component of the distributions of  $\underline{\bf X}_d(t)$ and $\underline{\bf Y}_{d}(t),t>0$, are projected on the subspaces $\mathbb{R}^m,1\leq m<d,$ and enter into the absolutely continuous part of $P\{\underline{\bf X}_{m}^d(t)\in d\underline{\bf x}_{m}\}$ and $P\{\underline{\bf Y}_{m}^d(t)\in d\underline{\bf y}_{m}\}$.
\begin{theorem}
For the random flights $\underline{\bf X}_{m}^d(t),t>0,$ and $\underline{\bf Y}_{m}^d(t),t>0,$ we have the following unconditional distributions
\begin{equation}\label{uncmarg}
\frac{P\{\underline{\bf X}_{m}^d(t)\in d\underline{\bf x}_{m}\}}{\prod_{j=1}^mdx_j}=\frac{(c^2t^2-||\underline{\bf x}_{m}||^2)^{\frac{d-m}{2}-1}}{\pi^{\frac m2}(ct)^{d-2}}\frac{E_{\frac {d-1}{2},\frac {d-m}{2}}\left(\frac{\lambda t(c^2t^2-||\underline{\bf x}_{m}||^2)^{\frac {d-1}{2}}}{(ct)^{(d-1)}}\right)}{E_{\frac{d-1}{2},\frac d2}(\lambda t)}
\end{equation}
with $d\geq 2,\,||\underline{\bf x}_{m}||<ct$ and
\begin{equation}\label{uncmarg2}
\frac{P\{\underline{\bf Y}_{m}^d(t)\in d\underline{\bf y}_{m}\}}{\prod_{j=1}^mdy_j}=\frac{(c^2t^2-||\underline{\bf y}_{m}||^2)^{\frac{d-m}{2}-1}}{\pi^{\frac m2}(ct)^{d-2}}\frac{E_{\frac {d}{2}-1,\frac {d-m}{2}}\left(\frac{\lambda t(c^2t^2-||\underline{\bf y}_{m}||^2)^{\frac {d}{2}-1}}{(ct)^{(d-2)}}\right)}{E_{\frac{d}{2}-1,\frac d2}(\lambda t)}
\end{equation}
with $d\geq 3,\,||\underline{\bf y}_{m}||<ct$.
\end{theorem}
\begin{proof}
We observe that the projection of the uniform distribution on the surface of $\mathcal{H}_{ct}^d$ onto $\mathbb{R}^{d-1}$ is obtained by means of the relationship $d\mathcal{H}_{ct}^d\sin\theta=\prod_{j=1}^{d-1}dx_j$, where $\sin\theta=\frac{\sqrt{c^2t^2-||\underline{\bf x}_{d-1}||^2}}{ct}$. Then, we obtain that
$$f_{\underline{\bf X}_{d-1}}^d(\underline{\bf x}_{d-1},t;0)
=\frac{\Gamma(\frac d2)}{(ct)^{d-1}\pi^{\frac d2}}\frac{\prod_{j=1}^{d-1}dx_j}{\sqrt{c^2t^2-||\underline{\bf x}_{d-1}||^2}}
$$
and by performing the following integrations
\begin{align*}
\int_{-\sqrt{c^2t^2-||\underline{\bf x}_{m}||^2}}^{\sqrt{c^2t^2-||\underline{\bf x}_{m}||^2}}dx_{m+1}\cdots\int_{\sqrt{c^2t^2-||\underline{\bf x}_{d-2}||^2}}^{\sqrt{c^2t^2-||\underline{\bf x}_{d-2}||^2}}f_{\underline{\bf X}_{d-1}}^d(\underline{\bf x}_{d-1},t;0)dx_{d-1}=\frac{\Gamma(\frac d2)}{\pi^{\frac m2}}\frac{(c^2t^2-||\underline{\bf x}_{m}||^2)^{\frac{d-m}{2}-1}}{(ct)^{d-2}\Gamma(\frac{d-m}{2})},
\end{align*}
with $ ||\underline{\bf x}_{m}||<ct$,
which corresponds to \eqref{eq:marg} and \eqref{eq:marg2} for $n=0$. Then, for the process $\underline{\bf X}_{m}^d(t)$ we obtain that
\begin{align*}
\frac{P\{\underline{\bf X}_{m}^d(t)\in d\underline{\bf x}_{m}\}}{\prod_{j=1}^mdx_j}
&=\sum_{n=0}^\infty f_{\underline{\bf X}_{m}}^d(\underline{\bf x}_{m},t,n)\,P\left\{\mathcal{N}_d(t)=n\right\}\notag\\
&=\frac{1}{\pi^{\frac m2}}\frac{1}{E_{\frac{d-1}{2},\frac d2}(\lambda t)}\sum_{n=0}^\infty\frac{(\lambda t)^n}{\Gamma(\frac {n+1}{2}(d-1)+\frac{1-m}{2})}\frac{(c^2t^2-||\underline{\bf x}_{m}||^2)^{\frac {n+1}{2}(d-1)-\frac{m+1}{2}}}{(ct)^{(n+1)(d-1)-1}}\notag\\
&=\frac{1}{\pi^{\frac m2}}\frac{(c^2t^2-||\underline{\bf x}_{m}||^2)^{\frac{d-m}{2}-1}}{(ct)^{d-2}}\frac{E_{\frac {d-1}{2},\frac {d-m}{2}}\left(\frac{\lambda t(c^2t^2-||\underline{\bf x}_{m}||^2)^{\frac {d-1}{2}}}{(ct)^{(d-1)}}\right)}{E_{\frac{d-1}{2},\frac d2}(\lambda t)}.
\end{align*}
with $||\underline{\bf x}_{m}||<ct$.

Analogous considerations on the random flight $\underline{\bf Y}_{m}^d(t),t>0,$ yield result \eqref{uncmarg2}.
\end{proof}

In Table \ref{uncproj}-\ref{uncproj2}, we sum up some particular important cases of the distributions \eqref{uncmarg} and \eqref{uncmarg2}.
 \begin{table}[h]
\begin{center}
\begin{tabular}{c|c|c}
&$d=2$&$d=3$\\		
\hline&&\\
$m=2$&-&$\frac{(c^2t^2-||\underline{\bf x}_{2}||^2)^{-\frac{1}{2}}}{\pi ct}\frac{E_{1,\frac {1}{2}}\left(\frac{\lambda t(c^2t^2-||\underline{\bf x}_{2}||^2)}{(ct)^2}\right)}{E_{1,\frac 32}(\lambda t)}$\\
$m=1$&$\frac{(c^2t^2-x_1^2)^{-\frac{1}{2}}}{\pi^{\frac 12}}\frac{E_{\frac{1}{2},\frac {1}{2}}\left(\frac{\lambda (c^2t^2-x_1^2)^{\frac{1}{2}}}{c}\right)}{E_{\frac{1}{2},1}(\lambda t)}$&$\frac{1}{\pi^{\frac 12}ct}\frac{e^{\frac{\lambda}{c^2t}(c^2t^2-x_1^2)}}{E_{1,\frac{3}{2}}(\lambda t)}$
\end{tabular}
\end{center}
\caption{Unconditional densities of $\underline{\bf X}_{m}^2(t),t>0,$ and $\underline{\bf X}_{m}^3(t),t>0$.}\label{uncproj}
\end{table}%

 \begin{table}[h]
\begin{center}
\begin{tabular}{c|c|c}
&$d=3$&$d=4$\\		
\hline&&\\
$m=3$&-&$\frac{\lambda t(c^2t^2-||\underline{\bf y}_{3}||^2)^{-\frac{1}{2}}}{(\pi)^{\frac32} (ct)^2}\frac{E_{1,\frac {1}{2}}\left(\frac{\lambda t(c^2t^2-||\underline{\bf y}_{3}||^2)}{(ct)^2}\right)}{e^{\lambda t}-1}$\\&&\\
$m=2$&
$\frac{(c^2t^2-||\underline{\bf y}_{2}||^2)^{-\frac{1}{2}}}{\pi ct}\frac{E_{\frac12,\frac {1}{2}}\left(\frac{\lambda t(c^2t^2-||\underline{\bf y}_2||^2)^{\frac12}}{ct}\right)}{E_{\frac12,\frac32}(\lambda t)}$&$\frac{\lambda t}{\pi (ct)^2}\frac{e^{\frac{\lambda}{c^2t}(c^2t^2-||\underline{\bf y}_{2}||^2)}}{e^{\lambda t}-1}$\\&&\\
$m=1$&$\frac{1}{\pi^{\frac 12}ct}\frac{E_{\frac{1}{2},1}\left(\frac{\lambda (c^2t^2-y_1^2)^{\frac{1}{2}}}{c}\right)}{E_{\frac{1}{2},\frac32}(\lambda t)}$&$\frac{\lambda t(c^2t^2-y_1^2)^{\frac{1}{2}}}{(\pi)^{\frac12} (ct)^2}\frac{E_{1,\frac {3}{2}}\left(\frac{\lambda t(c^2t^2-y_1^2)}{(ct)^2}\right)}{e^{\lambda t}-1}$
\end{tabular}
\end{center}
\caption{Unconditional densities of $\underline{\bf Y}_{m}^3(t),t>0,$ and $\underline{\bf Y}_{m}^4(t),t>0$.}\label{uncproj2}
\end{table}%
\begin{remark}
The projection onto the one-dimensional space $\mathbb{R}^1$ of \eqref{stadje} and of the singular component of a planar random flight with Poissonian times, becomes
\begin{equation}\label{struve}
\begin{split}
\overline p_{X_1}^2(x_1,t)
&=\frac{\lambda e^{-\lambda t}}{2c}\sum_{k=0}^\infty\left(\frac{\lambda}{2c}\sqrt{c^2t^2-x_1^2}\right)^{k-1}\frac{1}{\Gamma^2(\frac{k+1}{2})}\\
&=\frac{e^{-\lambda t}}{\pi\sqrt{c^2t^2-x_1^2}}+\frac{\lambda e^{-\lambda t}}{2c}\sum_{k=0}^\infty\left(\frac{\lambda}{2c}\sqrt{c^2t^2-x_1^2}\right)^k\frac{1}{\Gamma^2(\frac{k}{2}+1)}\\
&=\frac{e^{-\lambda t}}{\pi\sqrt{c^2t^2-x_1^2}}+\frac{\lambda e^{-\lambda t}}{2c}\left\{I_0\left(\frac{\lambda}{c}\sqrt{c^2t^2-x_1^2}\right)+{\bf L}_0\left(\frac{\lambda}{c}\sqrt{c^2t^2-x_1^2}\right)\right\},
\end{split}
\end{equation}
where $|x_1|\leq ct$,
and $I_0(x)=\sum_{k=0}^\infty\frac{(x/2)^2k}{(k!)^2}, x\in \mathbb{R},$ is the modified Bessel function, while ${\bf L}_0(x)=\sum_{k=0}^\infty\frac{( x/2)^{2k+1}}{(\Gamma(k+\frac32))^2}, $ $x\in \mathbb{R},$ is the modified Struve function. It is particularly interesting to compare \eqref{struve} with the probability distribution obtained in Table \ref{uncproj} for $d=2$ and $m=1$, namely
\begin{equation}\label{projecR}
\frac{P\{X_1^2(t)\in d x_1\}}{dx_1}=\frac{1}{\sqrt\pi\sqrt{c^2t^2-x_1^2}}\frac{E_{\frac{1}{2},\frac {1}{2}}\left(\frac{\lambda}{c} \sqrt{c^2t^2-x_1^2}\right)}{E_{\frac{1}{2},1}(\lambda t)}
=\frac{\lambda}{\sqrt{\pi}cE_{\frac{1}{2},1}(\lambda t)}\sum_{k=0}^\infty\left(\frac{\lambda}{c}\sqrt{c^2t^2-x_1^2}\right)^{k-1}\frac{1}{\Gamma(\frac{k+1}{2})}.
\end{equation}
with $ |x_1|\leq ct$.

\end{remark}
\begin{remark}
We observe that the $k$-th term, for $k\geq 2$, of \eqref{projecR} can be extracted from the uniform distribution inside the hypersphere $\mathcal{H}_{ct}^k=\{x_1,...,x_k:||\underline{\bf x}_k||\leq ct\}$ as follows
\begin{align*}
g_k(x_1)&=\frac{\Gamma(\frac k2)}{2\pi^{\frac k2}(ct)^k}\int_{-\sqrt{c^2t^2-x_1^2}}^{\sqrt{c^2t^2-x_1^2}}dx_2\cdots\int_{-\sqrt{c^2t^2-||\underline{\bf x}_{k-2}||^2}}^{\sqrt{c^2t^2-||\underline{\bf x}_{k-2}||^2}}dx_{k-1}\int_{-\sqrt{c^2t^2-||\underline{\bf x}_{k-1}||^2}}^{\sqrt{c^2t^2-||\underline{\bf x}_{k-1}||^2}}dx_{k}\\
&=\frac{\Gamma(\frac k2+1)}{\sqrt{\pi}\Gamma(\frac{k+1}{2})(ct)^k}\left(\sqrt{c^2t^2-x_1^2}\right)^{k-1}.
\end{align*}
The above distribution for $k=2$ yields the well-known Wigner law of which it represents an extension. We observe that
$$\int_{-ct}^{ct}x_1^{2m}g_k(x_1)dx_1=(ct)^{2m+k-1}\frac{\Gamma(m+\frac12)\Gamma(\frac k2+1)}{\sqrt{\pi}\Gamma(m+\frac k2+1)}$$
and for $k=2$ yields
$\int_{-ct}^{ct}x_1^{2m}g_2(x_1)dx_1=\binom{2m}{m}\frac{1}{m+1}\frac{1}{2^{2m+1}}2(ct)^{2m+1}$
involving the Catalan numbers and the distribution of the first return in the origin of the coin tossing process.

By summing up the distribution $g_k(x_1)$ with weighting terms represented by the fractional Poisson process $\mathcal{N}_2(t),t>0,$ with probability distribution
$$P\{\mathcal{N}_2(t)=k\}=\frac{(\lambda t)^k}{\Gamma(\frac k2+1)}\frac{1}{E_{\frac12,1}(\lambda t)},\quad k\geq 0,$$
we obtain the probability law \eqref{projecR}. By integrating the uniform law in the hypersphere $\mathcal{H}_{ct}^k$ with respect to the variables $x_{m+1},...,x_k,$ we obtain an $m$-dimensional extension of $g_k(x_1)$ in the following manner
\begin{align}\label{kterm2}
g_k(x_1,...,x_m)&=\frac{\Gamma(\frac k2)}{2\pi^{\frac k2}(ct)^k}\int_{-\sqrt{c^2t^2-||\underline{\bf x}_{m}||^2}}^{\sqrt{c^2t^2-||\underline{\bf x}_{m}||^2}}dx_{m+1}\cdots\int_{-\sqrt{c^2t^2-||\underline{\bf x}_{k-2}||^2}}^{\sqrt{c^2t^2-||\underline{\bf x}_{k-2}||^2}}dx_{k-1}\int_{-\sqrt{c^2t^2-||\underline{\bf x}_{k-1}||^2}}^{\sqrt{c^2t^2-||\underline{\bf x}_{k-1}||^2}}dx_{k}\notag\\
&=\frac{\Gamma(\frac k2+1)}{\pi^{\frac m2}\Gamma(\frac{k-m}{2}+1)(ct)^k}\left(\sqrt{c^2t^2-||\underline{\bf x}_m||^2}\right)^{k-m},
\end{align}
with $ k\geq m$.
For $m=k-1$, \eqref{kterm2} represents the $m$-dimensional extension of Wigner law. 
\end{remark}

 \section{On three-dimensional random flights governed by a Poisson process}
 
  The space $\mathbb{R}^3$ is an environment particularly important for the representation of the real motions. Therefore, we will focus here our attention on the random flights developing in the three-dimensional Euclidean space.

For the analysis developed in this Section, it is useful to observe that the Dirichlet distribution 
\eqref{eq:jointdis2} is related to the Poisson process and thus permits us to give an interesting interpretation of the random flight. If $T_1,...,T_n$ are the random instants at which the events of a homogeneous Poisson process occur it is well-known that 
 \begin{equation}\label{uniftime}
 P\{T_1\in dt_1,...,T_n\in dt_n|N(t)=n\}=\frac{n!}{t^n}dt_1\cdots dt_n\, {\bf 1}_{\{0<t_1<t_2<\cdots<t_n<t\}}
 \end{equation}
 By integrating \eqref{uniftime} as follows
\begin{align*}
&\frac{P\{T_{d-1}\in dt_{d-1},T_{2(d-1)}\in dt_{2(d-1)},...,T_{n(d-1)}\in dt_{n(d-1)}|N(t)=(n+1)(d-1)-1\}}{dt_{d-1}dt_{2(d-1)}\cdots dt_{n(d-1)}}\notag\\
&=\frac{((n+1)(d-1)-1)!}{t^{(n+1)(d-1)-1}}\int\limits_{\{0\leq t_1<\cdots<t_{d-1}<\cdots<t_{n(d-1)}<\cdots<t_{(n+1)(d-1)-1}<t\}} dt_1\cdots dt_{n(d-1)+1}\cdots dt_{(n+1)(d-1)-1}\notag\\
&=\frac{((n+1)(d-1)-1)!}{t^{(n+1)(d-1)-1}}\int\limits_{\{0\leq t_1<\cdots <t_{d-2}<t_{d-1}\}}dt_1\cdots dt_{d-2}\int\limits_{\{t_{d-1}< t_{d}<\cdots <t_{2(d-1)-1}<t_{2(d-1)}\}}dt_d\cdots dt_{2(d-1)-1}\notag\\
&\quad\cdots\int\limits_{\{t_{n(d-1)}< t_{n(d-1)+1}<\cdots <t_{(n+1)(d-1)-1}<t\}}dt_{n(d-1)+1}\cdots  dt_{(n+1)(d-1)-1}\notag\\
&=\frac{((n+1)(d-1)-1)!}{t^{(n+1)(d-1)-1}}\frac{t_{d-1}^{d-2}}{(d-2)!}\frac{(t_{2(d-1)}-t_{d-1})^{d-2}}{(d-2)!}\cdots\frac{(t-t_{n(d-1)})^{d-2}}{(d-2)!}\\
&=\frac{\Gamma((n+1)(d-1))}{(\Gamma(d-1))^{n+1}t^{(n+1)(d-1)-1}}\prod_{j=1}^{n+1}\tau_j^{d-2}
\end{align*}
with $\tau_j=t_{j(d-1)}-t_{(j-1)(d-1)},\,j=1,...,n+1,$ we obtain the law \eqref{eq:jointdis2}. For some details on these calculations see Lachal {\it et al.} (2006). The derivation of the Dirichlet distribution given here permits us to describe the random flight as a motion in $\mathbb{R}^d$ where a Poisson process governs the changes of orientation every $d-1$ events. In other words, every $d-1$ events of the Poisson process the moving particle changes direction ignoring all the previous $d-2$ events. For spaces of even dimension an analogous derivation of the second form \eqref{eq:jointdis3} of the Dirichlet law can be envisaged.

We introduce a random motion in $\mathbb{R}^3$ slightly different from that introduced in the previous Section. We suppose that the changes of direction are governed by an homogenous Poisson process. In particular, we assume that the particle changes direction (uniformly distributed on the surface of the sphere) only at even-valued Poisson events. Therefore, if the number of Poisson events is $N(t)=2n+1,\,n\geq 1$, we have that the position of the particle at time $t$ is represented by the following vector
\begin{align*}
&U_3(t)=c\sum_{k=1}^{n+1}(t_{2k}-t_{2k-2})\sin\theta_{2k-2}\sin\phi_{2k-2}=c\sum_{k=1}^{n+1}\tau_k\sin\theta_{2k-2}\sin\phi_{2k-2}\\
&U_2(t)=c\sum_{k=1}^{n+1}(t_{2k}-t_{2k-2})\sin\theta_{2k-2}\cos\phi_{2k-2}=c\sum_{k=1}^{n+1}\tau_k\sin\theta_{2k-2}\cos\phi_{2k-2}\\
&U_1(t)=c\sum_{k=1}^{n+1}(t_{2k}-t_{2k-2})\cos\theta_{2k-2}=c\sum_{k=1}^{n+1}\tau_k\cos\theta_{2k-2}
\end{align*}
where $t_k$ is the instant at which the $k$-th Poisson event happens with $t_0=0$ and $t_{2k+2}=t$. In view of the above considerations, for $d=3$, one has that
\begin{align*}
\frac{P\{T_2\in dt_2,T_4\in dt_4,...,T_{2n}\in dt_{2n}|N(t)=2n+1\}}{dt_2 dt_4\cdots dt_{2n}}
=\frac{(2n+1)!}{t^{2n+1}}\prod_{j=1}^{n+1}(t_{2j}-t_{2j-2})
\end{align*}
and then, the random flight $\underline{\bf U}_{3}(t)=(U_1(t),U_2(t),U_3(t)),t>0,$ has conditional characteristic function given by
\begin{align*}
&E\left\{e^{i<\underline{\bf \alpha}_{3},\underline{\bf U}_{3}(t)>}|N(t)=2n+1\right\}\\
&=\frac{(2n+1)!}{t^{2n+1}}\int_0^tt_2dt_2\cdots\int_{t_{2k-2}}^t(t_{2k}-t_{2k-2})(t-t_{2k})dt_{2k}\int_0^{2\pi}d\phi_0\cdots\int_0^{2\pi}\phi_{2n}\int_0^{\pi}d\theta_0\cdots\int_0^{2\pi}\theta_{2n}\notag\\
&\times\exp\bigg\{i\alpha_3c\sum_{k=1}^{n+1}(t_{2k}-t_{2k-2})\sin\theta_{2k-2}\sin\phi_{2k-2}+i\alpha_2c\sum_{k=1}^{n+1}(t_{2k}-t_{2k-2})\sin\theta_{2k-2}\cos\phi_{2k-2}\notag\\
&+i\alpha_1c\sum_{k=1}^{n+1}(t_{2k}-t_{2k-2})\cos\theta_{2k-2}\bigg\}\frac{\sin\theta_0\sin\theta_2\cdots\sin\theta_{2n}}{(4\pi)^{n+1}}=\text{(by using \eqref{intangle} for $d=3$)}\notag\\
&=\frac{(2n+1)!}{t^{2n+1}}\left(\frac{\pi}{2}\right)^{\frac{n+1}{2}}\int_0^tt_2dt_2\int_{t_2}^t(t_4-t_2)dt_4\cdots\int_{t_{2k-2}}^t(t_{2k}-t_{2k-2})(t-t_{2k})dt_{2k}\prod_{k=1}^{n+1}\frac{J_{\frac12}(c(t_{2k}-t_{2k-2})||\underline{\bf \alpha}_{3}||)}{\sqrt{c(t_{2k}-t_{2k-2})||\underline{\bf \alpha}_{3}||}}\notag\\
&=\frac{(2n+1)!}{t^{2n+1}}\left(\frac{\pi}{2}\right)^{\frac{n+1}{2}}\Bigg\{\int_0^t\tau_1d\tau_1\int_0^{t-\tau_1}\tau_2d\tau_2\cdots\int_0^{t-\sum_{k=1}^{n-1}\tau_k}\tau_n (t-\sum_{k=1}^{n}\tau_k)d\tau_n \prod_{k=1}^{n+1}\frac{J_{\frac12}(c\tau_k||\underline{\bf \alpha}_{3}||)}{\sqrt{c\tau_k||\underline{\bf \alpha}_{3}||}}\Bigg\}
\end{align*}
\begin{align*}
&=\frac{(2n+1)!}{t^{2n+1}}\left(\frac{\pi}{2}\right)^{\frac{n+1}{2}}\left\{\frac{2^{\frac n2+n+1}\Gamma(n+\frac32)t^{n+\frac12}}{\pi^{\frac{n+1}{2}}\Gamma(2n+2)(c||\underline{\bf \alpha}_{3}||)^{n+\frac12}}J_{n+\frac12}(ct||\underline{\bf \alpha}_{3}||)\right\}=\frac{2^{n+\frac12}\Gamma(n+\frac32)}{(ct||\underline{\bf \alpha}_{3}||)^{n+\frac12}}J_{n+\frac12}(ct||\underline{\bf \alpha}_{3}||)\notag
\end{align*}
where in the last step we have used the same approach as that developed in the proof of the result \eqref{cf}. Hence, the characteristic function of $\underline{\bf U}_{3}(t),t>0,$ (conditionally on the event $N(t)=2n+1$) coincides with \eqref{cf} for $d=3$. Then, by inverting the Fourier transform $E\left\{e^{i<\underline{\bf \alpha}_{3},\underline{\bf U}_{3}(t)>}|N(t)=2n+1\right\}$, we immediately obtain that
\begin{equation}\label{rfpoi}
\frac{P\{\underline{\bf U}_{3}(t)\in d\underline{\bf u}_{3}|N(t)=2n+1\}}{\prod_{j=1}^3du_j}=\frac{\Gamma(n+\frac 32)}{\pi^{\frac32}\Gamma(n)(ct)^{2n+1}}(c^2t^2-||\underline{\bf u}_3||^2)^{n-1},\quad n\geq 1,
\end{equation}
with $||\underline{\bf u}_3||<ct$, and coincides with the distribution \eqref{condlaw} for $d=3$. For $n=1$, the result \eqref{rfpoi} shows that we have an uniform distribution inside the sphere $\mathcal{H}_{ct}^3$. If $n=0$, that is $N(t)=1$, the random flight $\underline{\bf U}_{3}(t),t>0,$ (changing direction only at even-valued Poisson events) reaches the surface of $\mathcal{H}_{ct}^3$ with probability 
$
P\{\underline{\bf U}_{3}(t)\in \partial \mathcal{H}_{ct}^3\}=\lambda t e^{\lambda t}.
$

In this case, we can provide the following unconditional probability law
 \begin{align*}\frac{  P\left\{\underline{\bf U}_{3}(t)\in d\underline{\bf u}_{3},\bigcup_{n=1}^\infty (N(t)=2n+1)\right\}}{\prod_{j=1}^3du_j}
&=\sum_{n=1}^\infty P\{N(t)=2n+1\}\frac{P\{\underline{\bf U}_{3}(t)\in d\underline{\bf u}_{3}|N(t)=2n+1\}}{\prod_{j=1}^3du_j}\\
  &=\frac{e^{-\lambda t}}{\pi^{\frac 32}}\sum_{n=1}^\infty\frac{(\lambda t)^{2n+1}}{(2n+1)!}\frac{\Gamma(n+\frac 32)}{\Gamma(n)(ct)^{2n+1}}(c^2t^2-||\underline{\bf u}_3||^2)^{n-1}
\\
 &= \frac{e^{-\lambda t}}{\pi^{\frac 32}}\sum_{n=1}^\infty\frac{(\lambda t)^{2n+1}}{(2n+1)!}\frac{\sqrt{\pi}\Gamma(2n+2)2^{-2n-1}}{\Gamma(n+1)\Gamma(n)(ct)^{2n+1}}(c^2t^2-||\underline{\bf u}_3||^2)^{n-1}\\
 &= \frac{e^{-\lambda t}}{\pi}\left(\frac{\lambda}{2c}\right)^2\frac{1}{\sqrt{c^2t^2-||\underline{\bf u}_3||^2}}\sum_{n=1}^\infty\frac{\left(\frac{\lambda}{2c}\sqrt{c^2t^2-||\underline{\bf u}_3||^2}\right)^{2n-1}}{\Gamma(n+1)\Gamma(n)}\notag\\
  &= \frac{e^{-\lambda t}}{\pi}\left(\frac{\lambda}{2c}\right)^2\frac{1}{\sqrt{c^2t^2-||\underline{\bf u}_3||^2}}I_1\left(\frac{\lambda}{c}\sqrt{c^2t^2-||\underline{\bf u}_3||^2}\right)\notag
 \end{align*}
 
  Moreover, we obtain the projection of the absolutely continuous component of the distribution of $\underline{\bf U}_{3}(t),t>0,$ onto the plane as follows
  \begin{align*}
    \frac{P\{U_{1}(t)\in du_1,U_2(t)\in du_2|N(t)=2n+1\}}{du_1du_2}&=\int_{-\sqrt{c^2t^2- ||\underline{\bf u}_{2}||^2}}^{\sqrt{c^2t^2- ||\underline{\bf u}_{2}||^2}}\frac{P\{\underline{\bf U}_{3}(t)\in d\underline{\bf u}_{3}|N(t)=2n+1\}}{du_1du_2}\\
&= \frac{n+\frac12}{\pi(ct)^{2n+1}}
(c^2t^2- ||\underline{\bf u}_{2}||^2)^{n-\frac12}
  \end{align*}
for $n\geq 1$, while the projection of the uniform distribution on the surface $ \partial \mathcal{H}_{ct}^3$ onto the plane reads
$$ \frac{P\{U_{1}(t)\in du_1,U_2(t)\in du_2|N(t)=1\}}{du_1du_2}=\frac{1}{2\pi ct}\frac{1}{\sqrt{c^2t^2- ||\underline{\bf u}_{2}||^2}}.$$

 Then
 \begin{align*}
 & \frac{  P\{U_{1}(t)\in du_1,U_2(t)\in du_2,\bigcup_{n=0}^\infty (N(t)=2n+1)\}}{du_1du_2}\\
  &=\sum_{n=0}^\infty P\{N(t)=2n+1\} \frac{P\{U_{1}(t)\in du_1,U_2(t)\in du_2|N(t)=2n+1\}}{du_1du_2}\\
  &=\frac{e^{-\lambda t}}{\pi}\sum_{n=0}^\infty \frac{(\lambda t)^{2n+1}}{(2n+1)!}\frac{n+\frac12}{(ct)^{2n+1}}
(c^2t^2- ||\underline{\bf u}_{2}||^2)^{n-\frac12}
=\frac{e^{-\lambda t}}{\pi\sqrt{c^2t^2- ||\underline{\bf u}_{2}||^2}}\frac{\lambda}{2c}\sum_{n=0}^\infty \frac{1}{(2n)!}\left(\frac{\lambda}{c}\right )^{2n}
(\sqrt{c^2t^2- ||\underline{\bf u}_{2}||^2})^{2n}\\
&=\frac{\lambda e^{-\lambda t}}{2c\pi\sqrt{c^2t^2- ||\underline{\bf u}_{2}||^2}}\cosh\left(\frac{\lambda}{c}\sqrt{c^2t^2- ||\underline{\bf u}_{2}||^2}\right)
\end{align*}

We note that $q=q(u_1,u_2,t)= P\{U_{1}(t)\in du_1,U_2(t)\in du_2,\bigcup_{n=0}^\infty (N(t)=2n+1)\}$ is a solution to the planar telegraph equation
\begin{equation}
\frac {\partial^2 q}{\partial t^2}+2\lambda\frac {\partial q}{\partial t}=c^2\left\{\frac {\partial^2 }{\partial u_1^2}+\frac {\partial^2 }{\partial u_2^2}\right\}q
\end{equation}

Furthermore, we have that
  \begin{align*}
\frac{P\{U_{1}(t)\in du_1|N(t)=2n+1\}}{du_1}
&=\int_{-\sqrt{c^2t^2- u_1^2}}^{\sqrt{c^2t^2-u_1^2}}\frac{P\{U_{1}(t)\in du_1,U_2(t)\in du_2|N(t)=2n+1\}}{du_1}\\
&=  \frac{(2n+1)!}{(n!)^2(2ct)^{2n+1}}(c^2t^2-u_1^2)^n
  \end{align*}
  and then 
\begin{align}\label{protel}
  &\frac{  P\{U_1(t)\in du_1,\bigcup_{n=0}^\infty (N(t)=2n+1) \}}{du_1}
  = \frac{  P\{T(t)\in du_1,\bigcup_{n=0}^\infty (N(t)=2n+1)\}}{du_1}\\
  &=\sum_{n=0}^\infty P\{N(t)=2n+1\}\frac{P\{U_1(t)\in dx_1|N(t)=2n+1\}}{du_1}
  =e^{-\lambda t}\sum_{n=0}^\infty\frac{(\lambda t)^{2n+1}}{(n!)^2}\frac{1}{(2ct)^{2n+1}}(c^2t^2-u_1^2)^n\notag\\
  &=e^{-\lambda t}\frac{\lambda}{2c}I_{0}\left(\frac{\lambda}{2 c}\sqrt{c^2t^2- u_1^2}\right)
 =\frac{  P\{T(t)\in du_1,\bigcup_{n=0}^\infty (N(t)=2n+1) \}}{du_1}\notag
\end{align}
The result \eqref{protel} shows that the projection of $\underline{\bf U}_{3}(t),t>0,$ onto the real line is equivalent in distribution to the classical telegraph process $T(t),t>0.$

If the Poisson events recorded in $[0,t]$ are $2n$, we are not able to express in closed-form $P\{\underline{\bf U}_{3}(t)\in d\underline{\bf u}_{3}|N(t)=2n\}$.
Indeed, we have that 
 \begin{align*}
& \frac{P\{\underline{\bf U}_{3}(t)\in d\underline{\bf u}_{3}|N(t)=2n\}}{\prod_{j=1}^3du_j}\\
&=\frac{1}{(2\pi)^3}\int_{\mathbb{R}^3}e^{-i<\underline{\bf \alpha}_3,\underline{\bf u}_3>}E\left\{e^{i<\underline{\bf \alpha}_3,\underline{\bf U}_3(t)>}|N(t)=2n\right\}\prod_{j=1}^3d\alpha_j\\
  &=\frac{1}{(2\pi)^3}\int_{\mathbb{R}^3}e^{-i<\underline{\bf \alpha}_3,\underline{\bf u}_3>}\prod_{j=1}^3d\alpha_j\left(\frac\pi2\right)^{\frac{n+1}{2}}\frac{(2n)!}{t^{2n}}\int_0^t \tau_1d\tau_1\cdots\int_0^{t-\sum_{j=1}^{n-1}\tau_j}\tau_nd\tau_n\prod_{j=1}^{n+1}\frac{J_{\frac12}(c\tau_j||\underline\alpha_3||)}{\sqrt{c\tau_j||\underline\alpha_3||}}\\
  &=\frac{\left(\frac\pi2\right)^{\frac{n+1}{2}}}{(2\pi)^{\frac32}}\frac{(2n)!}{t^{2n}}\frac{1}{\sqrt{||\underline u_3||}}\int_0^\infty\rho^{\frac32}J_{\frac12}(\rho ||\underline u_3||)d\rho\int_0^t \tau_1d\tau_1\cdots\int_0^{t-\sum_{j=1}^{n-1}\tau_j}\tau_nd\tau_n\prod_{j=1}^{n+1}\frac{J_{\frac12}(c\tau_j||\underline\alpha_3||)}{\sqrt{c\tau_j||\underline\alpha_3||}}
 \end{align*}
Hence, to evaluate explicitely the distribution  $P\{\underline{\bf U}_{3}(t)\in d\underline{\bf u}_{3}|N(t)=2n\}$, we need to calculate the integrals of the following form
\begin{align*}
\int_0^a\frac{\sqrt{x}}{\sqrt{a-x}}J_{\frac12}(x)J_{\frac12}(a-x)dx&=\sum_{m=0}^\infty\sum_{r=0}^\infty\frac{(-1)^{m+r}a^{2m+2r+2}}{m!r!\Gamma(m+\frac12+1)\Gamma(r+\frac12+1)}\frac{1}{2^{2r+2m+1}}\int_0^ay^{2m+1}(a-y)^{2r}dy\\
&=\sum_{m=0}^\infty\sum_{r=0}^\infty\frac{(-1)^{m+r}a^{2m+2r+2}}{m!r!\Gamma(m+\frac12+1)\Gamma(r+\frac12+1)}\frac{1}{2^{2r+2m+1}}\frac{\Gamma(2m+2)\Gamma(2r+1)}{\Gamma(2(m+r)+3)}\\
&=\frac{2}{\pi}\sum_{m=0}^\infty\sum_{r=0}^\infty\frac{(-1)^{r+m}a^{2(m+r)+2}}{(2r+1)\Gamma(2(m+r)+3)}\\
&=\frac{2}{\pi}\sum_{m=0}^\infty\sum_{r=m}^\infty\frac{(-1)^{r}a^{2r+2}}{(2(r-m)+1)\Gamma(2r+3)}
=\frac{2}{\pi}\sum_{r=0}^\infty\frac{(-1)^{r}a^{2r+2}}{\Gamma(2r+3)}\sum_{m=0}^r\frac{1}{2m+1}
\end{align*}
 Unluckily, the above integral cannot be worked out explicitly and then the recursive approach used in the proof of Theorem \ref{th1} cannot be applied. Nevertheless, if $N(t)=2$, since $J_\frac12(x)=\sqrt{\frac{2}{\pi x}}\sin x$, we have that
 \begin{align*}
 P\{\underline{\bf U}_{3}(t)\in d\underline{\bf u}_{3}|N(t)=2\}=\frac{du_1du_2du_3}{\pi^2(ct)^2}\frac{1}{||\underline {\bf u}_3||}\int_0^t\frac{d\tau_1}{t-\tau_1}\int_0^\infty\frac{\sin(\rho||\underline{\bf u}_3||)\sin(c\tau_1\rho)\sin(c(t-\tau_1)\rho)}{\rho}d\rho
 \end{align*}
 The integral with respect to $\rho$ can be treated as in Orsingher and De Gregorio (2007). Indeed, since
\begin{equation*}
\sin x\sin y\sin z=\frac{1}{2^2}[\sin
(z+x-y)+\sin(z-x+y)-\sin(z+x+y)-\sin(z-x-y)]
\end{equation*}
and in light of the remarkable fact that
\begin{equation*}
\int_0^\infty\frac{\sin(A\rho)}{\rho}d\rho=
\begin{cases}
\frac{\pi}{2},& A>0,\\
-\frac{\pi}{2},& A<0,
\end{cases}
\end{equation*}  
the integral becomes
\begin{align*}
&\frac{1}{2^2}\int_0^\infty\frac{d\rho}{\rho}\bigg\{\sin(\rho(ct-||\underline{\bf u}_3||))+\sin(\rho(ct-2cs_1+||\underline{\bf u}_3||))
-\sin(\rho(ct+||\underline{\bf u}_3||))-\sin(\rho(ct-2cs_1-||\underline{\bf u}_3||))\bigg\}\\
&=\frac{\pi}{2^2}{\bf 1}_{\left[\frac{ct-||\underline {\bf u}_3||}{2c},\frac{ct+||\underline {\bf u}_3||}{2c}	\right]}(\tau_1)
\end{align*}
Then
 \begin{equation*}
 P\{\underline{\bf U}_{3}(t)\in d\underline{\bf u}_{3}|N(t)=2\}=\frac{\prod_{j=1}^{3}du_j}{\pi(2ct)^2}\frac{1}{||\underline {\bf u}_3||}\log\left(\frac{ct+||\underline{\bf u}_3||}{ct-||\underline{\bf u}_3||}\right),\quad ||\underline{\bf u}_3||\leq ct,
 \end{equation*}
 which corresponds to the law obtained in $\mathbb{R}^3$, for $n=1$, in Orsingher and De Gregorio (2007).

\end{document}